\documentclass[a4paper, 11p]{article}

\usepackage{amsmath}
\usepackage{amssymb}
\usepackage{amsthm}
\usepackage{bbm}

\numberwithin{equation}{section}
\author{Edgar Assing \\ assing@math.uni-bonn.de \\  Universit\"at Bonn}
\title{Adelic Vorono\"i summation and subconvexity for $GL(2)$ $L$-functions in the depth aspect}

\theoremstyle{plain}
\newtheorem{lemma}{Lemma}[section]
\newtheorem{prop}{Proposition}[section]
\newtheorem{theorem}{Theorem}[section]

\newtheorem{rem}{Remark}[section]

\DeclareMathOperator{\R}{\mathbb{R}}
\DeclareMathOperator{\C}{\mathbb{C}}
\DeclareMathOperator{\A}{\mathbb{A}}
\DeclareMathOperator{\Q}{\mathbb{Q}}
\DeclareMathOperator{\Z}{\mathbb{Z}}
\DeclareMathOperator{\N}{\mathbb{N}}

\DeclareMathOperator{\p}{\mathfrak{p}}

\DeclareMathOperator{\sgn}{\text{sgn}}

\newcommand{\abs}[1]{\left\vert #1 \right\vert}
\newcommand{\du}[1]{\boldsymbol{#1}}
\newcommand{\nilp}[1]{\left(\begin{matrix} 1&#1 \\ 0&1\end{matrix}\right)}
\newcommand{\cent}[1]{\left(\begin{matrix} #1&0 \\ 0&#1\end{matrix}\right)}
\newcommand{\am}[1]{\left(\begin{matrix} #1&0 \\ 0&1\end{matrix}\right)}

\begin{document}

\maketitle

\begin{abstract}
In this paper we establish a very flexible and explicit Vorono\"i summation formula. This is then used to prove an almost Weyl strength subconvexity result for automorphic $L$-functions of degree two in the depth aspect. That is, looking at twists by characters of prime power conductor. This is the natural $p$-adic analogue to the well studied $t$-aspect.
\end{abstract}

\tableofcontents

\section{Introduction}

This paper adds another result to the vast family of subconvex bounds for $L$-functions. However, we not only generalize a quite recent  subconvexity result for degree two $L$-functions, we also work out a very versatile version of the Vorono\"i summation formula which hopefully has other applications in the future. Before we state our results we will give a brief introduction to the subconvexity problem for automorphic $L$-functions.

Let $L(s)$ be an $L$-function in the sense of \cite{IK04} and let $C(s)$ be its analytic conductor. Then the Pharagm\'en-Lindel\"of principle implies the bound
\begin{equation}
	L(\frac{1}{2}+it) \ll_{\epsilon} C(\frac{1}{2}+it)^{\frac{1}{4}+\epsilon}. \label{eq:conv_bound}
\end{equation}
Due to the nature of the Pharagm\'en-Lindel\"of principle this bound is commonly referred to as convexity bound. The subconvexity problem for $L(s)$, in its most general form, is the problem of improving upon \eqref{eq:conv_bound} in the exponent. The best possible bound one may hope for is
\begin{equation}
	L(\frac{1}{2}+it) \ll_{\epsilon} C(\frac{1}{2}+it)^{\epsilon}. \nonumber
\end{equation}
This is known as the Lindel\"of conjecture and is a corollary of the Riemann Hypothesis for $L(s)$. 

While there are very little results towards the subconvexity problem for general $L$-functions, there is a huge amount of literature dealing with special cases and special families. For example, the subconvexity problem for automorphic $L$-functions of $GL_2$ over number fields has been solved completely, with non-specific exponent, in the ground breaking work \cite{MV10}. On the other hand, it has become a big business to obtain best possible numerical values for the exponent. Establishing strong subconvex bounds in a single aspect of the analytic conductor or for special automorphic $L$-functions has become a benchmark for the tools in use. Examples for such developments are the following. In \cite{Bo17} the bound
\begin{equation}
	\zeta(\frac{1}{2}+it)  \ll_{\epsilon} (1+\abs{t})^{\frac{13}{84}+\epsilon}   \nonumber
\end{equation}
demonstrates the strength of the decoupling method. This might be thought of as the $t$-aspect (or archimedean aspect) of the subconvexity problem for a very special $L$-function. A possible $p$-adic version of this has been considered in \cite{Mi16}. There it has been shown that
\begin{equation}
	L(\frac{1}{2},\chi) \ll_{\epsilon,p} q^{0.1645+\epsilon} \nonumber
\end{equation}
for a Dirichlet character $\chi$ of level $q=p^n$. This has been achieved by introducing an elaborate treatment of $p$-adic exponential pairs. The two bounds discussed so far are numerically very strong but work only for a very limited family of degree one $L$-functions. One out of many results concerning $L$-functions of $GL_2$ is
\begin{equation}
	L(\frac{1}{2}+it,f)  \ll_{f,\epsilon} (1+\abs{t})^{\frac{1}{3}+\epsilon}\nonumber
\end{equation}
for a holomorphic modular form $f$ of full level. This is initially due to Good \cite{Go82}. Another proof was later supplied by Jutila \cite{Ju87}. Recently the family to which this bound applies was enlarged by \cite{BMN17}. Indeed, the authors, relax the assumption on $f$ in the sense that they allow arbitrary level and central character. The $p$-adic analogue of this problem was considered by Blomer and Mili\'cevi\'c in \cite{BM15}. They show that 
\begin{equation}
	L(\frac{1}{2}+it,\chi\otimes f) \ll_{f,\epsilon} (1+\abs{t})^{\frac{5}{2}}p^{\frac{7}{6}}q^{\frac{1}{3}}, \nonumber
\end{equation}
where $f$ is a holomorphic or Maa\ss\  cuspidal newform of full level and $\chi$ is a Dirichlet character modulo $q=p^n$ for $p>2$. Our contribution to the subconvexity problem, similarly to the one in \cite{BMN17}, is to widen the family for which the above estimate holds. We will show the following.

\begin{theorem}\label{th:main_th_1}
Let $f$ be a cuspidal holomorphic or Maa\ss\ 	newform of level $N$ and central character $\omega$. Furthermore, let $\chi$ be a Dirichlet character of conductor $l^n$ for some prime $l$ satisfying $(l,2N)=1$ and $n\geq 5$. Then
\begin{equation}
	L(\frac{1}{2}+it,\chi\otimes f) \ll_{f,\epsilon} (1+\abs{t})^{\frac{5}{2}}l^{\frac{7}{6}+(\frac{1}{3}+\epsilon)n}. \nonumber
\end{equation}
\end{theorem}

As in \cite{BM15} this result will follow from a more general estimate for smooth sums of Hecke eigenvalues of automorphic forms. We will now state this result and refer to Subsection~\ref{sec:not} below for notation that was not yet introduced.

\begin{theorem}\label{th:main_th_2}
Let $l$ be an odd prime, $n_l\geq 10$ even, and $\pi$ be a cuspidal automorphic representation of conductor $Nl^{n_l}$ such that the $l$th component $\pi_l$ of $\pi$ is isomorphic to $\chi_l\abs{\cdot}_l^{\kappa_1}\boxplus \chi_l\abs{\cdot}_l^{\kappa_2}$ for some character $\chi_l\colon \Q_l^{\times}\to \C^{\times}$ of conductor $\frac{n_l}{2}$. Further, let $F$ be a smooth function with support in $[1,2]$ that satisfies $F^{(j)}\ll Z^j$ for some $Z\geq 1$. Then
\begin{equation}
	 L := \sum_{n\in \Z} \lambda_{\pi}(m)F\left( \frac{n}{M}\right) \ll_{\pi,\epsilon} Z^{\frac{5}{2}}l^{\frac{7}{6}} M^{\frac{1}{2}}l^{(\frac{1}{6}+\epsilon)n_l}, \nonumber
\end{equation}
for all $M\geq 1$ and all $\epsilon>0$.
\end{theorem}

We will prove this in Section~\ref{sec:proof_of_main_th} below, following exactly the same strategy as in \cite{BM15}. The novelty, which makes our generalization work, is a new version of the Vorono\"i summation formula. Such formulae play an important role in modern number theory, see \cite{MS04} for a very nice introduction. Our approach to Vorono\"i summation is based on ideas outlined in \cite{Te_vor}. The result is a very technical formula stated in Theorem~\ref{th:vor_with_Cs} below. The upshot is that we do not need any coprimality conditions between the denominator of the additive twist and the level of the automorphic form. A similar summation formula, with a different proof, has been used in \cite{BMN17}.

There are several natural generalizations of Theorem~\ref{th:main_th_2} that come to mind. Indeed, with a bit more work one should be able to relax the prescribed shape at the place $l$. Indeed it seems possible to deal with $\pi_l = \chi \pi_0$ for some fixed twist-minimal representation $\pi_0$ of $GL_2(\Q_l)$ and some non-trivial character $\chi$. 

Another interesting aspect would be to optimize the $N$ dependence in Theorem~\ref{th:main_th_2}. In our estimates we have been very wasteful in that aspect and we have thus included the $N$ dependence in the implied constant. However, one might be able to use explicit evaluations of ramified Whittaker new vectors in order to get the $N$-dependence into a reasonable range.

Finally, it is clear how to adapt our approach to the Vorono\"i summation formula to the number field setting. It would certainly be interesting to see if it is possible to work out a version of Theorem~\ref{th:main_th_2} over number fields.

\textbf{Acknowledgements.} I would like to thank A. Booker for suggesting this problem. I also want to thank A. Corbett for many valuable discussions on this and related topics. Finally I thank the referee for many comments which helped to improve this paper significantly.

\subsection{Notation and prerequisites} \label{sec:not}

Throughout this paper we will only consider the base field $\Q$. Its places, including the archimedean place $\infty$, are usually denoted by $p$. Each place $p$ comes with the local field $\Q_p$. For $p<\infty$ these fields are non-archimedean and we denote their ring of integers by $\Z_p$, and the unique maximal ideal by $\p$. We choose uniformizers $\varpi_p$ and normalize the absolute value by $\abs{\varpi_{p}}_p=p^{-1}$. The corresponding valuation on $\Q_p$ is determined by $v_p(\varpi)=1$. Further, we equip the local fields $\Q_p$ with two measures. First of all, we consider the Haar measure $\mu_p$ on $(\Q_p,+)$. If $p<\infty$,  these measures will be normalized such that $\mu_p(\Z_p)=1$. On $\Q_{\infty}=\R$ we take $\mu$ to be the standard Lebesgue measure. The second measure is the Haar measure $\mu_p^{\times}$ on $(\Q_p^{\times},\times)$. If $p<\infty$, it is explicitly given by $\mu_p^{\times}=\frac{\zeta_p(1)}{\abs{\cdot}_p}\mu_p$, where $\zeta_p(s)= (1-p^{-s})^{-1}$ denotes the local Euler factor of the Riemann zeta function. In particular, one has $\mu_p^{\times}(\Z_p^{\times}) = 1$. At the archimedean place $\infty$ we simply choose $\mu_{\infty}^{\times} = \frac{\mu_{\infty}}{\abs{\cdot}_{\infty}}$. The adele ring (resp. idele ring) over $\Q$ will be denoted by $\A$ (resp. $\A^{\times}$) and is equipped with the product measure $\mu$ (resp. $\mu^{\times}$).

We fix additive characters $\psi_p$ on $\Q_p$ such that the global additive character $\psi = \otimes_p \psi_p$ is $\Q$-invariant. Furthermore, at $p=\infty$ we take $\psi_{\infty}(x) = e(x) = e^{2\pi i x}$ and we assume that $\psi_p$ is trivial on $\Z_p$ but non-trivial on $\p^{-1}$ for every $p<\infty$. For a Schwartz-Bruhat function $f\in \mathcal{S}(\Q_p)$ we define the $p$-adic Fourier transform by
\begin{equation}
	\widehat{f}(y) = c_p\int_{\Q_p}f(x)\psi_p(xy)d\mu(x), \quad c_p = \begin{cases} 1 &\text{ if } p<\infty,\\ \frac{1}{\sqrt{2\pi}} &\text{ if } p = \infty. \end{cases} \nonumber
\end{equation} 
Note that our measures are normalized to be self-dual with respect to $\psi_p$.

For $p<\infty$ let ${}_p\mathfrak{X}$ denote the set of all multiplicative characters $\mu\colon \Q_p^{\times}\to S^1$ such that $\mu(\varpi_p)=1$. We also write ${}_p\mathfrak{X}_n$ (resp. ${}_p\mathfrak{X}_n'$) for the set of characters $\mu\in {}_p\mathfrak{X}$ with exponent-conductor $a(\mu)\leq n$ (resp $a(\mu)=n$). At the archimedean place we define  ${}_{\infty}\mathfrak{X}=\{1, \sgn\}$. These are exactly those characters $\mu\colon \R^{\times}\to S^1$ which are trivial on $\R_+$. Every quasi-character $\mu \colon \Q_p^{\times}\to \C^{\times}$ can be decomposed as $\mu=\abs{\cdot}_p^{t}\mu_0$ for some $t\in \C$ and some $\mu_0\in {}_p\mathfrak{X}$.  A global homomorphism $\chi \colon \Q^{\times}\setminus \A^{\times}\to \C^{\times}$ will be called a Hecke character. Note that each $\mu\in {}_l\mathfrak{X}$ induces a Hecke character $\chi_{\mu}$ defined by $\chi_{\mu} = \prod_p \chi_{\mu,p}$ with $\chi_{\mu,\infty}=\sgn^{\frac{1-\mu(-1)}{2}}$ and 
\begin{equation}
	\chi_{\mu,p}(ap^k) = \begin{cases}
		\mu(a) &\text{ if } p=l,\\
		\mu(p)^{-k} &\text{ if } p\neq l.
	\end{cases} \nonumber
\end{equation}
for $p<\infty$, $a\in \Z_p^{\times}$.  

A very useful tool is the $p$-adic logarithm $\log_p$, which can be defined on the set $1+\p\subset \Z_p$ via the well known Taylor series of the logarithm. As in the archimedean setting the $p$-adic logarithm is useful in order to translate between multiplicative and additive oscillations. Indeed, for $\mu_p\in{}_p\mathfrak{X}_n'$, $\kappa>0$ and $x\in \Z_p$ we have
\begin{equation}
	\mu_p(1+\varpi_p^{\kappa} x) = \psi_p\left(\frac{\alpha_{\mu_p}}{\varpi_p^n}\log_p(1+\varpi_p^{\kappa}x)\right) \label{eq:char_log}
\end{equation}
for some $\alpha_{\mu_p}\in \Z_p^{\times}$. In particular, if $\kappa\geq \frac{n}{2}$, one can safely truncate the logarithm after the first term and obtain
\begin{equation}
	\mu_p(1+\varpi_p^{\kappa} x) = \psi_p\left(\frac{\alpha_{\mu_p}x}{\varpi_p^n}\right). \nonumber
\end{equation} 

Finally, it will be useful to have a shorthand notation to deal with several places at once. For every $M\in \N$ we define
\begin{equation}
	\zeta_{M}(s) = \prod_{p\mid M}\zeta_p(s),\   \abs{\cdot}_M = \prod_{p\mid M}\abs{\cdot}_p \text{ and } (m,M^{\infty})= \prod_{p\mid M} p^{v_p(m)}. \nonumber
\end{equation}
We also write $\du{\mu}$ for a $M$-tuple of characters $\mu_p\in {}_p\mathfrak{X}$. Since we can always complete the tuple to all $p$ by inserting the trivial character at the remaining places, we dropped $M$ from the notation. One evaluates these tuples as as follows:
\begin{equation}
	\du{\mu}(x) = \prod_{p\leq \infty} \mu_p(x_p) = \prod_{p\mid M} \mu_p(x_p). \nonumber
\end{equation}
It is important not to confuse these tuples with Hecke characters. However, we can define the associated Hecke character
\begin{equation}
	\chi_{\du{\mu}} = \prod_{p\leq \infty} \chi_{\mu_p}. \nonumber
\end{equation}

Let $R$ be a commutative ring with 1. In our case $R$ will be either $\Q$, $\Q_p$, or $\A$. We set $G(R) = GL_2(R)$ and define the subgroups
\begin{eqnarray}
	Z(R) = \left\{  z(r)=\cent{r} \colon r\in R^{\times} \right\},& \quad A(R) = \left\{  a(r)=\am{r} \colon r\in R^{\times} \right\},  \nonumber\\
	N(R) = \left\{  n(x)=\nilp{x} \colon x\in R \right\}& \text{ and } \quad B(R) = Z(R)A(R)N(R). \nonumber
\end{eqnarray}
We use the following compact subgroups of $G(R)$ which depend on the underlying ring $R$. Define
\begin{eqnarray}
	K_{p} &=& GL_2(\Z_{p}) \text{ for }p<\infty, \nonumber \\
	K_{\infty} &=&  O_2(\R) ,  \nonumber \\
	K&=& \prod_{p\leq \infty}K_p \subset G(\A). \nonumber 
\end{eqnarray}
At the non-archimedean places, $p<\infty$, we also need the congruence subgroups
\begin{eqnarray}
	K_{1,p}(n) = K_{p} \cap \left[ \begin{matrix} 1+\varpi_{p}^n \Z_{p}&\Z_{p} \\ \varpi_{p}^n\Z_{p} &\Z_{p}\end{matrix} \right]. \nonumber
\end{eqnarray}
Finally we denote the long Weyl element by
\begin{equation}
	w=\left(\begin{matrix} 0&1 \\ -1&0\end{matrix}\right). \nonumber
\end{equation}

Let us briefly describe the measures on the groups in use. Locally, we will stick to the measure convention from \cite{Sa15_2}. This means, we use the identifications $N(R) = (R,+)$, $A(R) = R^{\times}$, and $Z(R) = R^{\times}$ to transport the measures defined on the local fields to the corresponding groups. Further, we take $\mu_{K_{p}}$ to be the probability Haar measure on $K_{p}$.
Globally, we choose the product measure on $K$, $N(\A)$ and $A(\A)$ coming from the previously defined local measures. The measure on $G(\A)$, in Iwasawa coordinates, is given by
\begin{equation}
	\int_{Z(\A)\setminus G(\A)} f(g) d\mu(g) = \int_K \int_{\A^{\times}} \int_{N(\A)} f(na(y) k) d\mu_{N(\A)}(n) \frac{d\mu^{\times}_{\A^{\times}}(y)}{\abs{y}_{\A}} d\mu_K(k). \nonumber
\end{equation}

In this work $\pi$ will usually denote a cuspidal automorphic representation of $G(\A)$ with central (Hecke) character $\omega_{\pi}$. That is an irreducible constitute of the right regular representation on $L_{\text{cusp}}^2(G(\Q)\backslash G(\A), \omega_{\pi})$.  It is well known that we can factor
\begin{equation}
	\pi = \bigotimes_{p\leq \infty} \pi_p, \nonumber
\end{equation}
where $\pi_p$ are irreducible, admissible, unitary representations of $G(\Q_p)$. These local representations come with several invariants. For example, the log-conductor $n_p=a(\pi_p)$ and the local central character $\omega_{\pi,p}$. The  contragredient representation will be denoted by $\tilde{\pi}_p$. Note that $\tilde{\pi}_p = \omega_{\pi,p}^{-1}\pi_p$. Attached to $\pi_p$ there are the usual suspects $\epsilon(\frac{1}{2},\pi_p)$ and $L(s,\pi_p)$. The representations of $G$ over local fields are completely classified. More precisely, we know that each smooth unitary irreducible infinite dimensional representation $\pi$ of $G(\Q_p)$ belongs to one of the following families.

\begin{enumerate}
\item \textbf{Twists of Steinberg:} $\pi_p = \chi St$, for some unitary character $\chi$. In this case we have $\omega_{\pi,p} = \chi^2$ and $a(\pi_p) = \max(1,2a(\chi))$. Furthermore, the $L$-factor as well as the $\epsilon$-factor are given by
\begin{equation}
	L(s,\pi_p) = \begin{cases} L(s,\abs{\cdot}_p^{\frac{1}{2}}) &\text{ if }\chi = 1, \\ 1 &\text{ if } \chi \neq 1,   \end{cases} \text{ and } \epsilon(\frac{1}{2}, \pi_p) = \begin{cases} -1 &\text{ if } \chi = 1, \\  \epsilon(\frac{1}{2},\chi)^2 &\text{ if } \chi \neq 1. \end{cases} \nonumber
\end{equation}

\item \textbf{Principal series:} $\pi_p = \chi_1\boxplus\chi_2$, for characters $\chi_1$ and $\chi_2$. In particular, $a(\pi) = a(\chi_1)+a(\chi_2)$ and $\omega_{\pi,p} = \chi_1\chi_1$. Concerning the $L$-factor we know
\begin{equation}
	L(s,\pi_p) = L(s,\chi_1)L(s,\chi_2) \text{ and } \epsilon(\frac{1}{2},\pi_p) = \epsilon(\frac{1}{2},\chi_1)\epsilon(\frac{1}{2}, \chi_2). \nonumber
\end{equation}

\item \textbf{Supercuspidal representations:} If $\pi_p$ is supercuspidal then $L(s,\pi_p)=1$. The other invariants are slightly more difficult to describe. Since it is not necessary for this work we will not go into further detail. 
\end{enumerate}

This list can be extracted from \cite{JL70} and \cite{Sc02}. Note that the characters $\chi_1$, $\chi_2$ appearing in unitary principal series representations are usually unitary themselves. However, if $\chi_1\vert_{\Z_p^{\times}} = \chi_2\vert_{\Z_p^{\times}}$ one might encounter situations where $\abs{\chi_i(\varpi_p)}\neq 1$. In this case one is dealing with $p$-adic complementary series. Unfortunately we can not exclude these representations from our discussion as the Ramanujan conjecture for $G(\A)$ is not yet known in full generality.

To any automorphic representation $\pi$ we attach its (incomplete)-L-function
\begin{equation}
	L(s,\pi) = \prod_{p< \infty} L(s,\pi_p) = \sum_{n\in \N} \lambda_{\pi}(n) n^{-s} \text{ for }\Re(s)\gg 1. \nonumber
\end{equation}
This function has a meromorphic continuation and satisfies the functional equation
\begin{equation}
	L(s,\pi_{\infty})L(s,\pi) = \left(\prod_{p\leq \infty} \epsilon(s,\pi_p)\right) L(1-s,\tilde{\pi}_{\infty})L(1-s,\tilde{\pi}). \nonumber
\end{equation}
The conductor of $\pi$ is given by $\prod_{p<\infty}p^{a(\pi_p)}$. This is not to be confused with the analytic conductor of $\pi$ mentioned in the introduction.

It is well known that in our case each $\pi$ is generic. Thus, there exists a (unique) $\psi$-Whittaker model $\mathcal{W}(\pi)$. This allows us, after fixing a suitable normalization, to associate to each $\phi$ in the representation space of $\pi$ a Whittaker function $W_{\phi}\in\mathcal{W}(\pi,\psi)$. If $\phi \in L_{\text{cusp}}^2(G(\Q)\backslash G(\A), \omega_{\pi})$ is a cuspidal function transforming according to $\pi$ the associated Whittaker function is given by the global Fourier coefficient
\begin{equation}
	W_{\phi}(g) = \int_{N(\Q)\setminus N(\A)}\phi(ng)\psi^{-1}(n)dn. \nonumber
\end{equation}

The twist $\chi \pi$ of an automorphic representation $\pi$ by a Hecke character $\chi$ is also an automorphic representation. It has central character $\chi^2\omega_{\pi}$ and its local constitutes are given by $\chi_p \pi_p$. 

At last, we introduce two more notions. By $\pi^b$ we denote the automorphic representation obtained from $\pi$ by passing (essentially) to the contragredient at the places $p\mid b$. More precisely, 
\begin{equation}
	\pi^b = \chi_{\omega_{\pi,b}^{-1}}\pi, \quad \text{for } \omega_{\pi,b} =\prod_{p\mid b} \omega_{\pi,p}. \nonumber
\end{equation}
More generally, we define
\begin{equation}
	(\pi)_{\du{\mu}} = \chi_{\du{\mu}}\pi. \nonumber
\end{equation}
These constructions may seem quite artificial. However, they will prove useful later on. Even more, the first construction is closely related to the theory of Atkin-Lehner involutions for classical newforms.

\section{A Vorono\"i summation formula}

In this section we use the machinery of automorphic representations to produce a very flexible Vorono\"i-type formula. In particular we want to produce a summation formula which relates a smoothed sum of Hecke eigenvalues to a dual sum which involves Hecke eigenvalues of twisted automorphic forms. To this end let $\pi$ be a cuspidal automorphic representation with conductor $Nl^{n_l}$ and central character $\omega_{\pi}$. The $L$-function of the associated contragredient representation is given by
\begin{equation}
	L(s,\tilde{\pi}) =  \sum_{n\in \N} \overline{\lambda_{\pi}(n)}n^{-s}. \nonumber
\end{equation} 
Our summation formula will feature the following ingredients. The main objects of interest are the Hecke eigenvalues $\lambda_{\pi}(n)$. Furthermore, we will allow additive twists $\psi_{\infty}(\zeta_0 m)$ for $\zeta_0 \in \Q$ satisfying $v_l(\zeta_0)>0$. Finally, we fix smooth, compactly supported test functions $W_{\infty} \colon \R\to \C$ and $W_{l}\colon \Q_{l} \to \C$.

In the following we will build on the ideas described in \cite{Te_vor} to derive an explicit Vorono\"i summation formula which is well suited for our application to the subconvexity problem. On the way we will use results from \cite{As17_1} to treat the places dividing $N$. Our method to implement the $l$-adic test function $W_l$ owes a great deal to the work \cite{Co17}. 

The main theorems of this section are stated at its very end. The reason for this is, that one should view this chapter as a recipe for generating explicit Vorono\"i formulae. We start of with the following fundamental identity.

\begin{lemma} \label{lm:basic_whitt_id}
Let $\zeta\in \A$ and let $\phi$ be a cuspidal function transforming according to $\pi$. Then we have
\begin{equation}
	\sum_{\gamma\in\Q^{\times}} \psi(\gamma\zeta)W_{\phi}\left(\left(\begin{matrix} \gamma&0 \\ 0&1 \end{matrix}\right)\right) = \sum_{\gamma\in \Q^{\times}}\tilde{W}_{\phi}\left(\left(\begin{matrix} \gamma&0 \\ 0&1 \end{matrix}\right)\left(\begin{matrix} 1&0 \\-\zeta&1 \end{matrix}\right)\right). \label{eq:basic_voronoi}
\end{equation}
for $\tilde{W}_{\phi}(g) = W_{\phi}(w{}^tg^{-1})$.
\end{lemma}
This is essentially \cite[Theorem~3.1]{Te_vor}. Similar identities are also used in \cite[Section~3]{Co14}.
\begin{proof}
We start by writing down the Whittaker expansion for $\phi$ with respect to $\psi$:
\begin{eqnarray}
	\phi\left(\left(\begin{matrix} 1&\zeta \\ 0&1 \end{matrix}\right)\right)&=& \sum_{\gamma\in\Q^{\times}}  W_{\phi}\left(\left(\begin{matrix} \gamma&\gamma \zeta \\ 0&1 \end{matrix}\right)\right) = \sum_{\gamma\in \Q^{\times}}  W_{\phi}\left(\left(\begin{matrix} 1&\gamma \zeta \\ 0&1 \end{matrix}\right)\left(\begin{matrix} \gamma&0 \\ 0&1 \end{matrix}\right)\right) \nonumber\\
	&=& \sum_{\gamma\in \Q^{\times}} \psi(\gamma\zeta)W_{\phi}\left(\left(\begin{matrix} \gamma&0 \\ 0&1 \end{matrix}\right)\right).\nonumber
\end{eqnarray}

Then we observe that
\begin{equation}
	\phi\left(\left(\begin{matrix} 1&\zeta \\ 0&1 \end{matrix}\right)\right) = [\imath \phi] \left(\left(\begin{matrix} 1&0\\ -\zeta&1 \end{matrix}\right)\right). \nonumber
\end{equation}
where $\imath \phi(g) = \phi({}^tg^{-1})$. 

We finish the proof by writing down the Whittaker expansion of $\imath\phi$ with respect to $\psi$:
\begin{eqnarray}
	[\imath \phi] \left(\left(\begin{matrix} 1&0\\ -\zeta&1 \end{matrix}\right)\right) &=& \sum_{\gamma\in\Q^{\times}} W_{\imath\phi}\left(\left(\begin{matrix} \gamma &0\\ 0&1 \end{matrix}\right)\left(\begin{matrix} 1&0\\ -\zeta&1 \end{matrix}\right)\right)\nonumber \\
	 &=& \sum_{\gamma\in \Q^{\times}} \tilde{W}_{\phi}\left(\left(\begin{matrix} \gamma &0\\ 0&1 \end{matrix}\right)\left(\begin{matrix} 1&0\\ -\zeta&1 \end{matrix}\right)\right) .\nonumber
\end{eqnarray}	

It is an easy calculation to check $\tilde{W}_{\phi}= W_{\imath\phi}$. Indeed,
\begin{eqnarray}
	\tilde{W}_{\phi}(g) &=& W_{\phi}(w{}^tg^{-1}) = \int_{N(\Q)\setminus N(\A)}\phi(nw{}^tg^{-1})\overline{\psi(n)}dn \nonumber \\
	&=& \int_{N(\Q)\setminus N(\A)}\phi(w{}^tn^{-1}{}^tg^{-1})\overline{\psi(n)}dn = \int_{N(\Q)\setminus N(\A)}\phi(w{}^t(ng)^{-1})\overline{\psi(n)}dn\nonumber \\
	&=& \int_{N(\Q)\setminus N(\A)}\phi({}^t(ng)^{-1})\overline{\psi(n)}dn = W_{\imath\phi}(g). \nonumber
\end{eqnarray}
\end{proof}

We will now proceed by choosing $\zeta$ and $\phi$ such that the left hand side takes the desired shape. In our case this choice is motivated by our application to the subconvexity problem. The next step will be to compute the right hand side as explicit as possible.

\subsection{Setting up the left hand side}

We choose $\phi$ such that 
\begin{equation}
	W_{\phi} = \prod_{p\leq \infty} W_{\phi,p}
\end{equation}
is a pure tensor.  Thus, we can treat each place on its own. 

Since the Kirillov model of $\tilde{\pi}_{\infty}$ contains the space of Schwartz functions we can choose $W_{\phi,\infty}$ such that, for all $\gamma\in \R^{\times}$, we have
\begin{equation}
	W_{\phi,\infty}\left(\left(\begin{matrix} \gamma&0 \\ 0&1 \end{matrix}\right)\right) = \omega_{\tilde{\pi},\infty}(\gamma)\abs{\gamma}_{\infty}^{\frac{1}{2}}W_{\infty}(\gamma), \nonumber
\end{equation}
for $W_{\infty}\in \mathcal{C}_c^{\infty}(\R^{\times})$ with support in $\R_+$.

At all the finite places $p\nmid lN$ we choose $\phi$ such that $W_{\phi,p}$ is
 the spherical $\psi_{p}$-Whittaker new vector $W_{\tilde{\pi},p}$ of $\tilde{\pi}_{p}$ normalized such that $W_{\tilde{\pi},p}(1)=1$. Indeed, for $\gamma\in \Q_p^{\times}$,
\begin{eqnarray}
	W_{\phi,p}(a(\gamma)) =W_{\tilde{\pi},p}(a(\gamma )) = \begin{cases} \lambda_{\tilde{\pi}}(p^{v_p(\gamma)})\abs{\gamma}_{p}^{\frac{1}{2}} &\text{ if } v_{p}(\gamma)\geq 0, \\ 0 &\text{ else.} \end{cases} \nonumber
\end{eqnarray}

If $p\neq l$ divides the level $N$, we will consider three cases. Recall from \cite[Lemma~2.5]{Sa15_2} that, for $k\in \Z$ and $v\in \Z_p^{\times}$,
\begin{equation}
	W_{\tilde{\pi},p}(a(\varpi_{p}^kv)) = \begin{cases}\xi(\varpi_{p}^k)p^{-k} &\text{ if $k\geq 0$ and $\tilde{\pi}_{p}=\xi \otimes St$ with $a(\xi)=0$},\\
	\omega_{\tilde{\pi},p}(v)\chi_1(\varpi_{p}^{l})p^{-\frac{k}{2}} &\text{ if $k\geq 0$ and $\tilde{\pi}_{p}=\chi_1\boxplus\chi_2$ for $a(\chi_1) >a(\chi_2)=0$,}\\
	\omega_{\tilde{\pi},p}(v) &\text{ if $k=0$ and $L(\tilde{\pi}_{p},s)=1$ },\\
	0 &\text{ else,}\end{cases} \nonumber
\end{equation}
where $W_{\tilde{\pi},p}$ is the normalized  $\psi_{p}$-Whittaker new vector of $\tilde{\pi}_{p}$. We set
\begin{equation}
	W_{\phi,p}(g) = W_{\tilde{\pi},p}(g) \text{ if $p\neq l$ divides $N$.}\nonumber
\end{equation}

At the place $l$ we choose $\phi$ so that
\begin{equation}
	W_{\phi,l}(a(\gamma)) = \omega_{\tilde{\pi},l}(\gamma)\abs{\gamma}_{l}^{\frac{1}{2}}W_{l}(\gamma), \nonumber
\end{equation}
for a smooth (i.e. locally constant) function $W_l\colon \Q_l^{\times}\to \C$ with support in $\Z_p^{\times}$. As in the archimedean case this is possible because the Kirillov model of $\tilde{\pi}_l$ contains the space of Schwartz-Bruhat functions, which in this case are exactly the smooth compactly supported functions on $\Q_{l}^{\times}$. 

We still have to pick $\zeta$. We define $\zeta_{\infty}=0$ and set
\begin{equation}
	\zeta_{\text{fin}} = \left(\zeta_0, \zeta_0, \zeta_0, \cdots \right),  \nonumber
\end{equation}
for $\zeta_0\in \Q$ such that $v_l(\zeta_0)\geq 0$. With this choice we have
\begin{equation}
	\psi(\zeta m) = \psi_{\text{fin}}(\zeta_{\text{fin}}m) \psi_{\infty}(0) = \psi_{\text{fin}}(\zeta_{\text{fin}}m) = \psi_{\infty}\left(-\zeta_0 m\right). \nonumber
\end{equation}
for every $m\in \Q$.

We conclude that the left hand side of \eqref{eq:basic_voronoi} (with our choice of $\phi$) equals
\begin{equation}
	\sum_{\gamma\in\Q^{\times}} \psi(\gamma\zeta)W_{\phi}\left(\left(\begin{matrix} \gamma&0 \\ 0&1 \end{matrix}\right)\right) = \sum_{\substack{m\in \N}} \lambda_{\pi}\left(\frac{m}{(m,l^{\infty})}\right) \psi_{\infty}(-\zeta_0 m) W_{\infty}(m)W_{l}(m). \nonumber
\end{equation}

\subsection{Computing the right hand side}

With the choices made above Lemma~\ref{lm:basic_whitt_id} yields the identity
\begin{eqnarray}
	\sum_{\substack{m\in \N}} \lambda_{\pi}\left(\frac{m}{(m,l^{\infty})}\right) \psi_{\infty}(-\zeta_0 m) W_{\infty}(m)W_{l}(m) = \sum_{\gamma\in \Q^{\times}}\tilde{W}_{\phi}\left(\left(\begin{matrix} \gamma&0 \\ 0&1 \end{matrix}\right)\left(\begin{matrix} 1&0 \\-\zeta&1 \end{matrix}\right)\right). \nonumber
\end{eqnarray}
We want to compute the right hand side as explicit as possible. To this end we observe that
\begin{eqnarray}
	&&\tilde{W}_{\phi}\left(\left(\begin{matrix} \gamma&0 \\ 0&1 \end{matrix}\right)\left(\begin{matrix} 1&0 \\-\zeta&1 \end{matrix}\right)\right) = W_{\phi}\left(w\left(\begin{matrix} \gamma^{-1}&0 \\ 0&1 \end{matrix}\right)\left(\begin{matrix} 1&\zeta \\0&1 \end{matrix}\right)\right) \nonumber \\
	&& = W_{\phi}\left(\left(\begin{matrix} 1&0 \\ 0&\gamma^{-1} \end{matrix}\right)w\left(\begin{matrix} 1&\zeta \\0&1 \end{matrix}\right)\right) =   \prod_{p\leq \infty} W_{\phi,p} \left( \left(\begin{matrix} \gamma & 0 \\ 0 &1\end{matrix} \right)w\left(\begin{matrix} 1 & \zeta  \\ 0 &1\end{matrix} \right)\right).  \label{eq:factorisation_rhs} 
\end{eqnarray}
The first equality follows directly from the definition and the second one uses $\Q$-invariance of the central character. The upshot is that we can do the remaining computations place by place.

\subsubsection{The unramified places $p\nmid lN$}

In this case we have $W_{\phi,p}(a(\gamma)) = W_{\tilde{\pi},p}(a(\gamma))$ and $\tilde{\pi}_p$ is unramified. Thus, if $v_{p}(\zeta_{p})\geq 0$, we obtain
\begin{eqnarray}
	W_{\phi,p} \left( \left(\begin{matrix} \gamma & 0 \\ 0 &1\end{matrix} \right)w\left(\begin{matrix} 1 & \zeta_{p}  \\ 0 &1\end{matrix} \right)\right) &=& W_{\tilde{\pi},p} \left( \left(\begin{matrix} \gamma & 0 \\ 0 &1\end{matrix} \right)\right) \nonumber\\
	 &=& \begin{cases}
		\abs{\gamma}_{p}^{\frac{1}{2}}\lambda_{\tilde{\pi}}(p^{v_{p}(\gamma)}) &\text{ if } \gamma \in \Z_{p}, \\
		0 &\text{ else. } 
	\end{cases} \nonumber
\end{eqnarray}
If $v_{p}(\zeta_{p}) <0$, the simple computation
\begin{equation}
	w\left( \begin{matrix} 1&\zeta_{p} \\ 0&1 \end{matrix} \right) = \left( \begin{matrix} 1&0 \\ -\zeta_{p}&1 \end{matrix} \right)w = \left( \begin{matrix} 1&-\zeta_{p}^{-1} \\ 0&1 \end{matrix} \right)\left( \begin{matrix} \zeta_{p}^{-1} &0 \\ 0&-\zeta_{p} \end{matrix} \right)\left( \begin{matrix} 0&1 \\ 1&-\zeta_{p}^{-1} \end{matrix} \right)w \nonumber
\end{equation}
implies
\begin{equation}
	\left(\begin{matrix} \gamma & 0 \\ 0 &1\end{matrix} \right)w\left(\begin{matrix} 1 & \zeta_{p}  \\ 0 &1\end{matrix} \right) = \left( \begin{matrix} 1& -\gamma \zeta_{p}^{-1} \\ 0 &1 \end{matrix} \right)\left( \begin{matrix} \gamma \zeta_{p}^{-1}&0 \\ 0 &1 \end{matrix} \right)\left( \begin{matrix} -1& 0 \\ -1 &-\zeta_{p} \end{matrix} \right). \nonumber
\end{equation}
Thus we arrive at
\begin{eqnarray}
	&&W_{\phi,p} \left( \left(\begin{matrix} \gamma & 0 \\ 0 &1\end{matrix} \right)w\left(\begin{matrix} 1 & \zeta_{p}  \\ 0 &1\end{matrix} \right)\right) = \psi_{p}(-\gamma \zeta_{p}^{-1}) W_{\phi,p} \left(a( \gamma\zeta_{p}^{-1}) \left(\begin{matrix} -1 & 0 \\ -1 &-\zeta_{p}\end{matrix} \right)\right) \nonumber\\
	&&\quad = \psi_{p}(-\gamma \zeta_{p}^{-1})\omega_{\tilde{\pi},{p}}(-\zeta_{p}) W_{\phi,p} \left( a( \gamma\zeta_{p}^{-2}) \left(\begin{matrix} 1 & 0 \\ \zeta_{p}^{-1} &1\end{matrix} \right)\right).\label{eq:old_action}
\end{eqnarray}
By right-$K_{p}$-invariance, the expression above simplifies to
\begin{eqnarray}
	&&W_{\phi,p} \left( \left(\begin{matrix} \gamma & 0 \\ 0 &1\end{matrix} \right)w\left(\begin{matrix} 1 & \zeta_{p}  \\ 0 &1\end{matrix} \right)\right) = \psi_{p}(-\gamma \zeta_{p}^{-1})\omega_{\tilde{\pi},{p}}(-\zeta_{p}) W_{\tilde{\pi},p} \left( \left(\begin{matrix} \gamma\zeta_{p}^{-2} & 0 \\ 0 &1\end{matrix} \right) \right) \nonumber \\
	&&\qquad = \begin{cases}
		\psi_{p}(-\gamma\zeta_{p}^{-1})\omega_{\tilde{\pi},{p}}(-\zeta_{p})\abs{\frac{\gamma}{\zeta_{p}^2}}_{p}^{\frac{1}{2}}\lambda_{\tilde{\pi}}(p^{v_{p}(\gamma\zeta_{p}^{-2})}) &\text{ if } v_{p}(\gamma\zeta_{p}^{-2})\geq 0, \\
		0 &\text{ else.} 
	\end{cases} \nonumber
\end{eqnarray}

\subsubsection{The ramified non-archimedean places $p\mid N$, $p\neq l$}

We now turn to the ramified places. The presence of arbitrary additive twists requires a careful analysis of ramified Whittaker new vectors. We carry this out by using several results from \cite{As17_1}.    

For $t\in \Z$, $l\in\N_0$ and $v\in \Z_p^{\times}$ we define
\begin{equation}
	g_{t,l,v} = \left(\begin{matrix} \varpi_{p}^t & 0 \\ 0 &1\end{matrix} \right)w\left(\begin{matrix} 1 & v\varpi_{p}^{-l} \\ 0 &1\end{matrix} \right) = \left(\begin{matrix} 0 & \varpi_{p}^t \\ -1 & -v\varpi_{p}^{-l}\end{matrix} \right). \nonumber
\end{equation}
Then we observe that the matrix at which we want to evaluate $W_{\phi,p}$ is
\begin{equation}
	\left(\begin{matrix} \gamma & 0 \\ 0 &1\end{matrix} \right)w\left(\begin{matrix} 1 & \zeta_{p}  \\ 0 &1\end{matrix} \right) = \begin{cases}
		g_{t,0,u^{-1}} \left(\begin{matrix} 1 & \zeta_{p}-1 \\ 0 &u\end{matrix} \right) &\text{ if $v_{p}(\zeta_{p})\geq 0$ and $\gamma = u\varpi_{p}^t$,}  \\
		g_{t,l,u^{-1}v} \left(\begin{matrix} 1 & 0 \\ 0 &u\end{matrix} \right) &\text{ if $\zeta_{p}=v\varpi_{p}^{-l}$ and $\gamma = u\varpi_{p}^t$.}
	\end{cases} \nonumber
\end{equation}
Since the matrices on the right are always in $K_{1,p}(\infty)$ we can use the finite Fourier expansion (better known as $c_{t,l}(\mu)$-expansion) to calculate the value of $W_{\phi,p}$ explicitly. This has been studied extensively in \cite{As17_1}.

Let $n_{p}=v_{p}(N)$. Then we treat several subcases which feature different behavior. We set 
\begin{equation}
	N_0 = \prod_{\substack{p\mid N,\\ -v_{p}(\zeta_{p})\leq 0}} p^{n_{p}}, N_1 = \prod_{\substack{p\mid N,\\ 0<-v_{p}(\zeta_{p})<n_{p}}} p^{n_{p}} \text{ and }  N_2= \prod_{\substack{p\mid N,\\ n_{p}\leq -v_{p}(\zeta_{p})}} p^{n_{p}}. \label{eq:def_Ns}
\end{equation}

In order to use the results from \cite{As17_1} we have to re-normalize our representation $\pi_{p}$. To do so we fix an unramified character $\xi_{p}$ such that $\omega_{\xi_p^{-1}\pi_{p}}(\varpi_{p})=1$. 

If $p\mid N_0$, we have
\begin{eqnarray}
	&&W_{\phi,p} \left( a(\gamma)w\left(\begin{matrix} 1 & \zeta_{p}  \\ 0 &1\end{matrix} \right)\right) = \xi_{p}(\gamma) W_{\xi_{p}^{-1}\tilde{\pi}_{p}}(g_{v_{p}(\gamma),0,u^{-1}}) = \xi_{p}(\gamma)c_{v_{p}(\gamma),0}(1) \nonumber \\
	&& \qquad = \begin{cases}
		\epsilon(\frac{1}{2},\pi_{p}) \abs{\gamma \varpi_{p}^{n_{p}}}_{p}^{\frac{1}{2}} \lambda_{\tilde{\pi}}(p^{v_{p}(\gamma)+n_{p}}) &\text{ if } v_p(\gamma N)\geq 0, \\
		0 &\text{ else.} 
	\end{cases} \nonumber
\end{eqnarray}
This follows from the explicit evaluation of $c_{t,0}(1)$ given in \cite{As17_1}. For a complete classification of the constants $c_{t,l}(\mu_p)$ see Appendix~\ref{app:ctlv}.

\begin{rem}
The case $v_{p}(\zeta_{p})\geq 0$ at ramified places can be treated in general using the theory of Atkin-Lehner operators. This leads the same result as our $c_{t,0}(1)$ approach. See \cite[Section~6]{Te_vor} for details.
\end{rem}

If $p\mid N_1$ the situation is slightly more complicated. We define
\begin{equation}
	\mathcal{E}_{p}(\gamma,\zeta_{p}) = \abs{\gamma}_{p}^{-\frac{1}{2}}W_{\xi^{-1}_{p}\tilde{\pi}_{p}}(g_{t,l,u^{-1}v}). \label{eq:def_E}
\end{equation}
We have the following result towards the support of these coefficients.
\begin{lemma}
For $v_{p}(\gamma) < \min(2v_{p}(\zeta_{p}),-n_{p}+v_{p}(\zeta_{p}))$ we have
\begin{equation}
	\mathcal{E}_{p}(\gamma,\zeta_{p}) = 0. \nonumber
\end{equation}
\end{lemma}
\begin{proof}
This follows directly from the explicit formulas given in \cite[Lemma~3.1, 3.3, 3.4, 3.5 and 3.6]{As17_1}.
\end{proof}

Thus, we can write
\begin{eqnarray}
	W_{\phi,p} \left( a(\gamma)w\left(\begin{matrix} 1 & \zeta_{p}  \\  0 &1\end{matrix} \right)\right)  = \begin{cases}
		\abs{\gamma}_{p}^{\frac{1}{2}}\xi_{p}(\gamma)\mathcal{E}_{p}(\gamma,\zeta_{p}),  &\text{ if } v_{p}(\gamma) \geq -n_{p}+v_{p}(\zeta_{p}), \\ 
		0 &\text{ else.}
	\end{cases}\nonumber
\end{eqnarray}

If, for the global application, it is not necessary to keep track of $N_1$ dependence it can be useful to expand $\mathcal{E}_{p}(\gamma,\zeta_{p})$ in terms of $c_{t,l}(\mu_p)$. We will follow this path later on. However, if one is interested in keeping track of possible cancellation coming from these places, one has to work more carefully. In this scenario one can obtain completely explicit formulas involving $p$-adic oscillations if one evaluates $W_{\xi^{-1}_{p}\tilde{\pi}_{p}}$. Such evaluations have been given in \cite{As17_1} in several special cases.

Finally, if $p\mid N_2$, we make the following observation.
\begin{lemma} \label{lm:high_b_part}
Let $-\nu(\zeta_{p})\geq n_{p}$ then
\begin{eqnarray}
	&&W_{\tilde{\pi},p} \left( \left(\begin{matrix} \gamma & 0 \\ 0 &1\end{matrix} \right)w\left(\begin{matrix} 1 & \zeta_{p}  \\ 0 &1\end{matrix} \right)\right) \nonumber \\
	&&\quad  = \begin{cases}
		\omega_{\pi,p}\left(-\zeta_{p}\gamma^{-1}\right) \psi_{p}\left(-\gamma\zeta_{p}^{-1}\right)\abs{\gamma \zeta_{p}^{-2}}_{\nu}^{\frac{1}{2}}\lambda_{\pi}(p^{v_{\p}(\gamma \zeta_{p}^{-2})}), &\text{ if }v_{p}(\gamma\zeta_{p}^{-2})\geq 0,\\
		0 &\text{ else.}
	\end{cases} \nonumber
\end{eqnarray}
\end{lemma}
\begin{proof}
For $l\geq n$ we have the decomposition
\begin{equation}
	g_{t,l,v} \left(\begin{matrix}	0 & 1 \\ p^{n} & 0 \end{matrix}\right) = n(-v^{-1}p^{t+l})z(vp^{n-l})g_{t-n+2l,0,v^2}\left( \begin{matrix} 1 & 1+v^{-1}p^{l-n} \\ 0 & -v^{-2} \end{matrix} \right).\nonumber
\end{equation}
Thus, using  \cite[Lemma~2.17, Corollary~2.27, Proposition~2.28]{Sa15_2} the claimed expression is reduced to the evaluation of $c_{t-n+2l,0}(1)$ given in the appendix.
\end{proof}

This completes the treatment for ramified non-archimedean places away from $l$ for now.

\subsubsection{The special place $p=l$}

At this place we are dealing with a Whittaker function which is not necessarily a new vector. To evaluate this function away from the diagonal we will use the \textit{local functional equation}. Note that this approach enables us to include a non-archimedean test function at any place $l$. In particular, we can treat the case when $\pi$ ramifies at $l$.

We define
\begin{equation}
	Z(W,s,\mu_p)=\int_{\Q_{p}^{\times}}W(a(y))\mu_p(y)\abs{y}_{p}^{s-\frac{1}{2}}d^{\times}y, \nonumber
\end{equation}
for a multiplicative character $\mu_p\in{}_p\mathfrak{X}$, a Schwartz-Bruhat function $W$, and some complex number $s$ with sufficiently large real part. Then the local functional equation is
\begin{equation}
	\frac{Z(W,s,\mu_p)}{L(s,\mu_p\tilde{\pi}_{p})}\epsilon(s,\mu_p\tilde{\pi}_{p})=\frac{Z(\tilde{\pi}_p(w)W,1-s,\mu_p^{-1}\omega_{\tilde{\pi},p}^{-1})}{L(1-s,\mu_p^{-1}\pi_{p})}. \nonumber
\end{equation}
Since $\psi_{p}$ is unramified, we have
\begin{equation}
	\epsilon(s,\mu_p\tilde{\pi}_{p}) =p^{(\frac{1}{2}-s)a(\mu_p\tilde{\pi}_p)}\epsilon(\frac{1}{2},\mu_p\tilde{\pi}_p). \nonumber
\end{equation} 
The upshot is, that the latter $\epsilon$-factors have absolute value 1.

Recall that we want to evaluate
\begin{equation}
	W_{\phi,p} \left( \left(\begin{matrix} \gamma & 0 \\ 0 &1\end{matrix} \right)w\left(\begin{matrix} 1 & \zeta_{p}  \\ 0 &1\end{matrix} \right)\right). \nonumber
\end{equation}
Thus, we define $W=\tilde{\pi}_p(n(\zeta_p))W_{\phi,p}$ so that the local functional equation reads
\begin{eqnarray}
	&&\int_{\Q_{p}^{\times}}W_{\phi,p}(a(y)wn(\zeta_p))[\omega_{\tilde{\pi},p}\mu_p]^{-1}(y)\abs{y}_{p}^{\frac{1}{2}-s}d^{\times}y\nonumber \\
	&&\qquad = \epsilon(s,\mu_p\tilde{\pi}_{p})\frac{L(1-s,\mu_p^{-1}\pi_{p})}{L(s,\mu_p\tilde{\pi}_{p})}Z(W,s,\mu_p).\nonumber \nonumber
\end{eqnarray}
The latter $Z$-integral can be computed, because on the diagonal $W_{\phi,p}$ is given by $W_l$. To do so we will apply $p$-adic Mellin inversion to this formula. Recall that the Mellin transform is defined by
\begin{equation}
	[\mathfrak{M}f](\mu_p\abs{\cdot}_{p}^s) = [\mathfrak{M}f](\mu_p,s) = \int_{\Q^{\times}_p}f(y)\mu_p(y)\abs{y}_{p}^sd^{\times}y . \nonumber
\end{equation}
The inverse Mellin transform is given by
\begin{equation}
	[\mathfrak{M}^{-1}\tilde{f}](y) = \frac{\log(p)}{2\pi}\sum_{\mu_p\in {}_p\mathfrak{X}}\mu_p(y)^{-1}\int_{-\frac{\pi}{\log(p)}}^{\frac{\pi}{\log(p)}} \tilde{f}(\mu_p,it)\abs{y}_{p}^{-it}dt. \nonumber
\end{equation}
Indeed, see \cite[Proposition~7.1.4]{Co17}, these transforms satisfy
\begin{equation}
	\mathfrak{M}^{-1}\circ \mathfrak{M} = \mathfrak{M} \circ\mathfrak{M}^{-1} = 1. \nonumber
\end{equation}
It will be useful for us to split the inverse transform into two pieces. We define the \textit{pre-Mellin inversion} by
\begin{equation}
	[\mathfrak{M}_{\text{pre}}^{-1}\tilde{f}](\mu_p,y) = \frac{\log(p)}{2\pi}\int_{-\frac{\pi}{\log(p)}}^{\frac{\pi}{\log(p)}} \tilde{f}(\mu_p,it)\abs{y}_{p}^{-it}dt. \nonumber
\end{equation}
This leads to the definition
\begin{equation}
	B_{\pi_{p},\kappa}(y)=\frac{\log(p)}{2\pi}\int_{-\frac{\pi}{\log(p)}}^{\frac{\pi}{\log(p)}} \frac{L(1-\kappa-it,\tilde{\pi}_{p})}{L(\kappa+it,\pi_{p})}\abs{y}_{p}^{-it}dt. \nonumber
\end{equation}
Indeed, $B_{\pi_p}$ turns out to be a very valuable $p$-adic special function in this context. For example, if $L(s,\pi_p)=1=L(s,\tilde{\pi}_p)$, then
\begin{equation}
	B_{\pi_{p},\kappa}(y) = \mathbbm{1}_{\Z_{p}^{\times}}(y^{-1}), \quad\text{for every } \kappa\in\R. \label{eq:nice_support_for_simplereps}
\end{equation}
The other extreme appears for $\tilde{\pi}_p$ unramified. In this case we have
\begin{equation}
	B_{\pi_{p},\frac{1}{2}}(y)=\begin{cases}
		\abs{y}_{p}^{-\frac{1}{2}}P_{\pi}(-v_p(y)) & \text{ if $y^{-1}\in \Z_{p}$,}\\
		p^{-\frac{1}{2}}(\omega_{\pi,p}(\varpi_{p})p^{-1}\lambda_{\tilde{\pi}}(p)-\lambda_{\pi}(p)) &\text{ if $y^{-1}\in \varpi_{p}^{-1}\Z^{\times}_{p}$,} \\
		\omega_{\pi,p}(\varpi_{p})p^{-1} &\text{ if $y^{-1}\in \varpi_{p}^{-2}\Z_{p}^{\times}$},\\
		0 &\text{ else,}
	\end{cases} \label{eq:formula_unram_reps}
\end{equation}
with $$P_{\pi}(v_p(y)) = \lambda_{\tilde{\pi}}(p^{v_{p}(y)})-p^{-1}\lambda_{\pi}(p)\lambda_{\tilde{\pi}}(p^{1+v_{p}(y)})+\omega_{\pi,p}(\varpi_{p})p^{-2}\lambda_{\tilde{\pi}}(p^{2+v_{p}(y)}).$$

We are finally ready to evaluate $W_{\phi,p}$. The assumption $\text{supp}(W_{l})\subset\Z_{p}^{\times}$ makes our life a lot easier. Indeed,
\begin{equation}
	[\mathfrak{M}W](\mu_p,s) = [\mathfrak{M}W_{l}](\mu_p,0) \text{ for all} s\in \C. \nonumber
\end{equation}
Furthermore, we make the simplifying assumption $v_l(\zeta_p)\geq 0$.
The local functional equation set up as above reads
\begin{eqnarray}
	&&[\mathfrak{M} W_{\phi,p}(a(\cdot)wn(\zeta_p))](\mu_p, it) \nonumber \\
	&&\qquad = \epsilon(\frac{1}{2}-it,\mu_p^{-1} \pi_{p})\frac{L(\frac{1}{2}+it,\mu_p\tilde{\pi}_{p})}{L(\frac{1}{2}-it,\mu_p^{-1}\pi_{p})} \nonumber [\mathfrak{M}W_{l}](\mu_p,0). \nonumber
\end{eqnarray}
In this situation we can compute the pre-Mellin inversion explicitly in terms of $B_{\mu_p\tilde{\pi}_{p},\Re(s)}$.  After completing the process of Mellin inversion we arrive at
\begin{equation}
	W_{\phi,p}(a(y)wn(\zeta_p)) = \sum_{\mu_p\in {}_p\mathfrak{X}}\mu_p(y^{-1})B_{\mu_p^{-1}\pi_{p},\frac{1}{2}}(\varpi_{p}^{-a(\mu_p\tilde{\pi}_{p})}y^{-1})\epsilon(\frac{1}{2},\mu_p^{-1}\pi_{p})[\mathfrak{M}W_l^{\omega_{\pi,p}}](\mu_p). \nonumber
\end{equation}
This defines a $p$-adic version of the Hankel transform. We define
\begin{equation}
	\mathcal{H}W_{l}(y) = \abs{y}_{p}^{-\frac{1}{2}}\sum_{\mu_p\in {}_p\mathfrak{X}}\mu_p(y^{-1})B_{\mu_p^{-1}\pi_{p},\frac{1}{2}}(\varpi_{p}^{-a(\mu_p\tilde{\pi}_{p})}y^{-1})\epsilon(\frac{1}{2},\mu_p^{-1}\pi_{p})[\mathfrak{M}W_l](\mu_p). \label{eq:def_p_adic_bessel}
\end{equation} 
Thus, we have 
\begin{equation}
	W_{\phi,p}(a(y)wn(\zeta_{l})) =  \abs{y}_{p}^{\frac{1}{2}}\mathcal{H}W_{l}(y). \nonumber
\end{equation}
We will encounter a similar formula at the archimedean places.  The $p$-adic Hankel transform has the following properties.

\begin{lemma}\label{lm:properties_p_adc_Besel}
If for some $\kappa\geq 1$
\begin{equation}
	W_l(x+yl^{\kappa}) = W_l(x) \text{ for all } x\in \Z_l^{\times}, y\in \Z_l, \nonumber
\end{equation}
then one can restrict the $\mu_p$-sum in \eqref{eq:def_p_adic_bessel} to $\mu_p\in {}_l\mathfrak{X}_{\kappa}$. Furthermore,
\begin{equation}
	\text{supp}(\mathcal{H}W_l)\subset \varpi_l^{\min(-2\kappa, -a(\pi_l))}\Z_l. \nonumber	
\end{equation}
\end{lemma}
\begin{proof}
The first statement is a simple consequence of the following computation. For $\mu_p$ satisfying $a(\mu_p)>\kappa$ we have
\begin{equation}
	[\mathfrak{M}W_l](\mu_p) = \sum_{x\in \Z_{l}^{\times}/(1+l^{\kappa}\Z_l)} W_l(x)\mu_p(x) \int_{1+l^{\kappa}\Z_l} \mu_p(y)d^{\times} y = 0. \nonumber
\end{equation}
The second statement follows from the first one together with the support properties of $B_{\mu_p^{-1}\pi_{l},\frac{1}{2}}$.
\end{proof}

\subsubsection{The archimedean places}

At $\infty$ the action of the element $w$ in the archimedean Kirillov model is given by the Hankel transform:
\begin{eqnarray}
	&&W_{\phi, \infty} \left( \left(\begin{matrix} \gamma & 0 \\ 0 &1\end{matrix} \right)w\left(\begin{matrix} 1 & \zeta_{\infty}  \\ 0 &1\end{matrix} \right)\right) =  W_{\phi, \infty} \left( \left(\begin{matrix} \gamma & 0 \\ 0 &1\end{matrix} \right)w\right)  \nonumber \\
	&&\qquad =  \int_{\R^{\times}}  j_{\tilde{\pi},\infty}(x\gamma) \abs{x}_{\infty}^{\frac{1}{2}}W_{\infty}(x) \frac{dx}{x}. \nonumber 
\end{eqnarray}
The function $j_{\tilde{\pi},{\infty}}$ can be computed explicitly and it turns out that, if $\tilde{\pi}_{\infty}$ is a discrete series representation of weight $k\geq 2$ with central character $\text{sgn}^k$,
\begin{eqnarray}
	j_{\tilde{\pi},{\infty}}(y) = \begin{cases}
		2\pi i^k \sqrt{y} J_{k-1}(4\pi \sqrt{y}) &\text{ if } y>0, \\
		0 &\text{ if } y<0.
	\end{cases} \nonumber
\end{eqnarray}
If $\tilde{\pi}_{\infty}=\abs{\cdot}^{ir}\boxplus\abs{\cdot}^{-ir}$, then we have  
\begin{eqnarray}
	j_{\tilde{\pi},{\infty}}(y) = \begin{cases}
		i\pi\sqrt{y}\frac{J_{i2r}(4\pi\sqrt{y})-J_{-i2r}(4\pi\sqrt{y})}{\sinh(\pi r)}&\text{ if } y>0, \\
		4\cosh(\pi r)\sqrt{\abs{y}}K_{i2r}(4\pi \sqrt{\abs{y}}) &\text{ if } y<0.
	\end{cases} \nonumber
\end{eqnarray}
These expressions also hold for complementary series $\tilde{\pi}_{\infty}$, which appear when $r$ is imaginary.
To shorten notation later on we write 
\begin{equation}
\mathcal{H}_{\pm}W_{\infty}(y) = \int_{\R_{>0}}\mathcal{J}_{\infty,\kappa}^{\pm}(4\pi\sqrt{xy})W_{\infty}(x)dx, \label{eq:def_hankelarch}
\end{equation}
for $y>0$. Where we set
\begin{eqnarray}
	\mathcal{J}_{\infty,\kappa}^{\pm}(y) &=& \frac{4\pi}{y}j_{\tilde{\pi},\infty}\left(\pm\frac{y^2}{16\pi^2}\right)  \nonumber
\end{eqnarray}
and $\kappa$ is $k-1$ in the case of discrete series and $2r$ for principal series or complementary series. We choose this notation to be compatible with \cite{BM15}. In particular, at infinity, we have
\begin{eqnarray}
	W_{\phi,\infty} \left( \left(\begin{matrix} \gamma & 0 \\ 0 &1\end{matrix} \right)w\left(\begin{matrix} 1 & \zeta_{\infty}  \\ 0 &1\end{matrix} \right)\right) = \abs{\gamma}_{\infty}^{\frac{1}{2}} \mathcal{H}_{\sgn(\gamma)}W_{\infty}(\abs{\gamma}). \nonumber
\end{eqnarray}

\subsection{Summary}

The following proposition summarizes our findings from the previous subsections.

\begin{prop}\label{pr:rough_vor}
Let $\pi$ be a cuspidal automorphic representation with conductor $Nl^{n_l}$ and central character $\omega_{\pi}$. Furthermore, let $\frac{a}{b}\in \Q$, $W_{\infty}\in \mathcal{C}_0^{\infty}(\R_+)$, and for some prime $l\nmid b$  let $W_l\in \mathcal{S}(\Q_l)$ with support in $\Z_l^{\times}$. We define $N_0$, $N_1$ and $N_2$ as in \eqref{eq:def_Ns}. Furthermore, we set
\begin{eqnarray}
	b_1 &=& (b,N_1), \quad b_2 = (b,N_2^{\infty}), \quad b_0 = \frac{b}{b_1b_2}, \nonumber\\
	\eta(\pi,a,b) &=& \prod_{p\mid N_0} \epsilon(\frac{1}{2},\pi_{p})\prod_{p\mid b_0N_2} \omega_{\pi,p}(-ab), \text{ and }\nonumber \\
	\mathcal{E}(m,\frac{a}{b}) &=& \prod_{p\mid N_1 }\xi_p\left(\frac{m}{b_1N_1}\right) \mathcal{E}_p\left(\frac{m}{b_0^2b_2^2N_0b_1N_1},\frac{a}{b}\right), \nonumber
\end{eqnarray}
where the local function $\mathcal{E}_p$ is defined in \eqref{eq:def_E}. Then
\begin{eqnarray}
	&&\sum_{m\in\N} e\left(-\frac{a}{b}m\right)\lambda_{\pi}\left(\frac{m}{(m,l^{\infty})}\right)W_{\infty}(m)W_l(m)  \nonumber \\
	&&= \frac{\eta(\pi,a,b)}{b_0b_2\sqrt{N_0}}\sum_{c\in \Z}\sum_{\substack{m\in\Z,\\(m,l)=1}} e\left(l^cm\frac{\overline{a N_0N_1}}{b_0b_2}\right) \lambda_{\pi^{N_2}}\left(\frac{m}{(m,N_1^{\infty})}
	\right)\nonumber \\  
	&&\qquad \qquad \cdot\mathcal{H}_{\sgn(m)}W_{\infty}\left(\frac{l^c\abs{m}}{b_0^2b_2^2b_1N_1N_0}\right)\mathcal{H}W_{l}\left(\frac{l^c\abs{m}}{b_0^2b_2^2b_1N_1N_0}\right)\mathcal{E}\left(l^cm,\frac{a}{b}\right).\nonumber
\end{eqnarray}
\end{prop}

This proposition is already a very robust tool with many interesting features. However, it has the caveat that the contribution from the places $p\mid N_1$ is hidden in the mysterious term $\mathcal{E}$. In order to make our formula more suitable for applications we will now unfold this error using local Fourier analysis. 

\begin{theorem} \label{th:vor_with_Cs}
Under the assumptions of Proposition~\ref{pr:rough_vor} we have 
\begin{eqnarray}
	&&\sum_{m\in\N} e\left(-\frac{a}{b}m\right)\lambda_{\pi}\left(\frac{m}{(m,l^{\infty})}\right)W_{\infty}(m)W_l(m)  \nonumber \\
	&&=\zeta_{N_1}(1) \frac{\eta(\pi,a,b)}{b_0b_2\sqrt{N_0}}\sum_{c\in \Z}\sum_{\du{\mu}\in \prod_{p\mid b_1}{}_p\mathfrak{X}_{v_p(b_1)}} \frac{\du{\mu}(\frac{ab_0b_2N_0N'_1(\du{\mu})}{b_1l^c})}{\sqrt{b_1N_1'(\du{\mu})}}\sum_{m_1\mid N_1^{\infty}}C(\pi_{N_1},\du{\mu},b_1,m_1)\nonumber \\ 	&&\qquad \cdot \lambda_{(\pi^{N_2})_{\du{\mu}}}\left(m_1\frac{N_1'(\du{\mu})}{b_1N_1}\right)\sum_{\substack{m\in\Z,\\(m,lN_1)=1}} e\left(l^cm_1m\frac{\overline{a N_0N_1}}{b_0b_2}\right) \lambda_{(\pi^{N_2})_{\du{\mu}}}\left(m
	\right)\nonumber \\  
	&&\qquad \qquad \qquad \qquad\cdot\mathcal{H}_{\sgn(m)}W_{\infty}\left(\frac{l^cm_1\abs{m}}{b_0^2b_2^2b_1N_1N_0}\right)\mathcal{H}W_{l}\left(\frac{l^cm_1m}{b_0^2b_2^2b_1N_1N_0}\right).\nonumber
\end{eqnarray}
For some constants $C(\pi_{N_1},\du{\mu},b_1,m_1)\in \C$ satisfying
\begin{equation}
	\abs{C(\pi_{N_1},\du{\mu},b_1,m_1)} \ll_{N_1} m_1^{\frac{7}{64}+\epsilon}.\label{eq:estimat_C}
\end{equation}
\end{theorem}
\begin{proof}
The idea, taken from \cite[(11)]{Sa15_2}, is to expand
\begin{equation}
	W_{\xi^{-1}_p\tilde{\pi}_p}(g_{t,k,v}) = \sum_{\mu_p\in{}_p\mathfrak{X}_k}c_{t,k}(\mu_p)\mu_p(v), \nonumber
\end{equation}
for each $p\mid N_1$. The constants $c_{t,k}(\mu_p)$ depend on the underlying representation $\pi_p$ and have been described in Appendix~\ref{app:ctlv}. Using these expansions we can write
\begin{eqnarray} 
	\mathcal{E}(l^cm,\frac{a}{b}) &=&  \zeta_{N_1}(1)\sum_{\du{\mu}\in \prod_{p\mid b_1}{}_p\mathfrak{X}_{v_p(b_1)}} \frac{\du{\mu}(\frac{ab_0b_2N_1N_0}{ml^c})}{\sqrt{b_1N_1'(\du{\mu})}}    \lambda_{\chi_{\du{\mu}}\tilde{\pi}}\left(\frac{N_1'(\du{\mu})}{b_1N_1}m\right)   \nonumber \\
	&& \qquad\qquad \qquad \qquad\qquad \qquad \qquad \qquad \cdot C(\pi_{N_1},\du{\mu},b_1,(m,N_1^{\infty})), \nonumber
\end{eqnarray}
for
\begin{eqnarray}
	N_1'(\du{\mu}) &=& \prod_{p\mid N_1}p^{a(\mu_p\tilde{\pi}_p)+\delta_{\mu_p\tilde{\pi}_p}}, \nonumber \\
	C(\pi_{N_1},\du{\mu},b_1,(m,N_1^{\infty})) &=& \prod_{p\mid N_1} c_p\left(\tilde{\pi}_{p},v_{p}(b_1),v_p(\frac{m}{b_1N_1}),\mu_p\right)\xi_p^{-1}(N_1'(\du{\mu}))p^{\frac{\delta_{\mu_p\tilde{\pi}_p}}{2}}. \nonumber
\end{eqnarray}
Inserting this expression in Proposition~\ref{pr:rough_vor} completes the proof of the stated expression. The bound on the coefficients $C(\pi_{N_1},\du{\mu},b_1,m_1)\in \C$ can be read off from \eqref{eq:bound_cp} together with the current best possible results towards the Ramanujan conjecture. See for example \cite{BB11}.
\end{proof}

\section{Application to the subconvexity problem}\label{sec:proof_of_main_th}

In this section we will prove Theorem~\ref{th:main_th_2}. In doing so we will closely stick to \cite{BM15} and assume some familiarity with the arguments therein. From now on $\pi$ will denote a cuspidal automorphic representation of conductor $Nl^{n_l}$. We are interested in 
\begin{equation}
	L : =\sum_{m\in \Z} \lambda_{\pi}(m)F(\frac{m}{M}), \nonumber
\end{equation}
for a smooth function $F$ with support in $[1,2]$ satisfying $F^{(j)}\ll_j Z^j$ for some $Z\geq 1$. 

We will restrict our attention to  $\pi_l = \chi_l\abs{\cdot}^{\kappa_1}\boxplus \chi_l\abs{\cdot}^{\kappa_2}$ for some $\chi_l\in {}_l\mathfrak{X}_{\frac{n_l}{2}}'$ and $n_l\geq 10$ even. In particular, there is a Hecke character $\chi=\prod_{p\leq \infty} \chi_p$ such that $\pi = \chi \otimes \pi_0$ for some automorphic representation $\pi_0$ which is unramified at $l$. This implies that for all $(m,l)=1$ we have
\begin{equation}
	\lambda_{\pi}(m) = \chi_l(m)^{-1}\lambda_{\pi_0}(m). \label{eq:translating_osciallation}
\end{equation}

\subsection{Sketch of proof}

Before we start to estimate $L$ in detail let us briefly discuss the approach taken in \cite{BM15}. This should serve as orientation throughout the rest of this rather technical section. The approach is based on Jutila's method. For a more detailed discussion see \cite[Section~1 and~4]{BM15}.

The first step is to split $L$ into arithmetic progressions
\begin{equation}
		L  = \sum_{a,b,t} \sum_{\substack{m\in \Z, \\
		m\equiv \frac{b}{a}\alpha_{\chi_l}\  (l^t)}} \lambda_{\pi}(m)F(\frac{m}{M}), \nonumber
\end{equation} 
where $\alpha_{\chi_l}\in\Z_l^{\times}$ is defined by \eqref{eq:char_log}. On these arithmetic progressions the $p$-adic analytic oscillation of $\chi_l(m)$, which in turn models the $p$-adic behavior of $\lambda_{\pi}(m)$ via \eqref{eq:translating_osciallation}, can be modelled by an exponential with a well behaved phase. Furthermore, due to the $p$-adic Dirichlet approximation theorem given in \cite[Theorem~3]{BM15}, the fraction $\frac{a}{b}$ as well as the modulus $l^t$ are controlled by several parameters that will be chosen later. The result of this step is \eqref{eq:result_step_1} which corresponds to \cite[(5.5)]{BM15}. This can be thought of as a $p$-adic Farey dissection. Indeed, we are left with a double sum involving $p$-adic Farey fractions $\frac{a}{b}$ and an arithmetic progression $m\equiv \frac{b}{a}\alpha_{\chi_l} \ (l^t)$ of large modulus.

The second step is to dualise the $m$-sum by applying Vorono\"i summation. After evaluating the resulting $p$-adic and archimedean Fourier-type integrals using the method of stationary phase, this is the content of Lemma~\ref{lm:p-adic_stat} and Lemma~\ref{lm:arch_stat} below, we arrive at \eqref{eq:end_Step_2} which nicely compares to \cite[(5.15)]{BM15}. At this stage we are roughly left with an average of short exponential sums weighted by Hecke eigenvalues, where the average is taken over the Farey fractions $\frac{a}{b}$. Indeed, we need to estimate sums like
\begin{equation}
	\sum_{m}\sum_{a,b} \frac{\lambda_{\sigma}(m)}{m^{\frac{1}{4}}}\kappa_{a,b}(m), \nonumber
\end{equation} 
for  some at $l$ unramified twist $\sigma$ of $\pi_0$ and an oscillatory function $\kappa_{a,b}$. More precisely, the function $\kappa_{a,b}(\cdot)$ consists of an archimedean exponential with linear phase, some harmless weights and the $l$-adic oscillation defined in \eqref{eq:l-adic_oscilation} below. Note that it is this step where we use the Vorono\"i formula from Theorem~\ref{th:vor_with_Cs}, which produces some overhead coming from those Farey fractions where $(b,N)\neq 1$. This is the source for most of the technical notation in this section. However, the oscillatory integrals that arise at $l$ and $\infty$ are essentially those appearing in \cite{BM15} and we can reduce their evaluation to \cite[Lemma~3 and~4]{BM15}.

The third step is to apply the Cauchy-Schwarz inequality. As a result we are able to separate the automorphic weights from the oscillator term $\kappa_{a,b}$. After opening the square and exchanging order of summation we are left with something like
\begin{equation}
	\left(\sum_m \frac{\abs{\lambda(m)}^2}{m^{\frac{1}{2}}}\right)^{\frac{1}{2}}\left(\sum_{\substack{a_1,b_1,\\ a_2,b_2}}\sum_m \kappa_{a_1,b_1}(m)\overline{\kappa_{a_2,b_2}(m)}\right)^{\frac{1}{2}}. \nonumber
\end{equation}
Here the second $m$-sum is a short exponential sum depending on the Farey fraction $\frac{a_1}{b_1}$ and $\frac{a_2}{b_2}$. The precise formula, after eliminating the first $m$-sum using the Ranking-Selberg bound, appears in  \eqref{eq:part_of_step3} extending \cite[(5.17)]{BM15}. It is here where it is possible to extract cancellation on average. In particular, we obtain diagonal terms, those where $\abs{a_1b_1-a_2b_2}_l\ll 1$, that are estimated trivially and off-diagonal therms where extra cancellation will be obtained. Even though the idea is taken directly from \cite{BM15} we are still plagued by the luggage acquired in the previous step, which makes the implementation slightly technical.

The final step is to bound the exponential sums 
\begin{equation}
	\Xi_{a_1,a_2,b_1,b_2}=\sum_m\kappa_{a_1,b_1}\overline{\kappa_{a_2,b_2}}, \nonumber
\end{equation} 
for $\abs{a_1b_1-a_2b_2}\gg 1$, non-trivially. This will be done by utilizing a second derivative $p$-adic van der Corput estimate given in \cite[Theorem~5]{BM15}. The result of this analysis is summarised in Lemma~\ref{lm:after_vdC} which is our analogue of \cite[Lemma~5]{BM15}. We can reduce the proof to the result \cite[Lemma~13]{BM15} which establishes properties of the $p$-adic oscillation which are essential for the application of the second derivative test. After applying Lemma~\ref{lm:after_vdC} we are in the same situation as in \cite[Section~5.5]{BM15} which can be followed to conclude the proof.

We now start with a detailed execution of the strategy outlined above.

\subsection{$p$-Adic Farey dissection}

In this subsection we perform the first step described above. This relies on \cite[Theorem~3]{BM15} which we now recall.

\begin{theorem}
Let $\alpha\in \Z_{l}^{\times}$, $q\in \N$ and an integer $-q\leq r\leq q$ be given. Write $r^+=\max(r,0)$ and $r^-=\max(-r,0)$, and let
\begin{equation}
	S=\{ (a,b,k)\in \Z\times \N\times \N_0 \vert b\leq l^{k+2r^-}, \abs{a}\leq l^{k+2r^+}, (a,b)=(a,l)=(b,l)=1 \}. \nonumber
\end{equation}
For $(a,b,k)\in S$, let
\begin{equation}
	\Z_{l}^{\times}[a,b,k]=\{ m\in \Z_l^{\times} \vert b\alpha/m-a \in l^{q+\abs{r}+k}\Z_l\}. \nonumber
\end{equation}
Then there exists a subset $S^0\subset S$ such that
\begin{equation}
	\Z_l^{\times} = \bigsqcup_{(a,b,k)\in S^0} \Z_{l}^{\times}[a,b,k] \nonumber
\end{equation}
and in addition the following two properties hold: if $(a,b,k_1), (a,b,k_2)\in S^0$, then $k_1=k_2$, and for each $(a,b,k)\in S^0$ one has $k\leq q-\abs{r}$. 
\end{theorem}

Applying this theorem with $\alpha = \alpha_{\chi_l} \in \Z_l$ as defined in \eqref{eq:char_log}, $q\leq \frac{n_l}{8}$ and some $\abs{r}\leq q$ yields
\begin{equation}
	L = \sum_{s=(a,b,k)\in S^0} \sum_{m\in \Z\cap \Z_p^{\times}[a,b,k]}\lambda_{\pi}(m)F\left(\frac{m}{M}\right). \nonumber  
\end{equation}
For later reference we define
\begin{equation}
	L_s=\sum_{m\in \Z\cap \Z_p^{\times}[a,b,k]}\lambda_{\pi}(m)F\left(\frac{m}{M}\right). \nonumber
\end{equation}
We estimate
\begin{equation}
	L \ll l^{n_l\epsilon} \max_{\substack{0\leq k\leq q-\abs{r},\\ A\leq \frac{1}{2}l^{k+2r^+}, \\ B\leq \frac{1}{2}l^{k+2r^-} }} \abs{L_{A,B,k}}, \quad \text{for } L_{A,B,k}=\sum_{\substack{s=(a,b,k)\in S^0, \\ A\leq \abs{a}<2A, \\ B\leq b<2B}} L_s . \label{eq:result_step_1}
\end{equation}
Good bounds for $L_{A,B,k}$ will suffice to establish good (non-trivial) bounds for $L$. Thus, we fix $A$, $B$ and $k$ until otherwise stated. 

We define the $p$-adic test function
\begin{equation}
	W_l(s;x)=\mathbbm{1}_{\Z_l^{\times}[a,b,k]}(x) \chi_l(x)^{-1}\psi_l\left(\frac{a\overline{b}}{l^{\frac{n_l}{2}}}x\right) \label{eq:def_p_adic_testfk}
\end{equation}
and rewrite
\begin{equation}
	L_s = \sum_m \lambda_{\pi_0}(m)W_l(s;m) e\left(\frac{a\overline{b}}{l^{\frac{n_l}{2}}}m\right)F\left(\frac{m}{M}\right). \nonumber
\end{equation}
Here we use the fact, that 
\begin{equation}
	\psi_{l}\left(\frac{a\overline{b}}{l^{\frac{n_l}{2}}}m\right)e\left(\frac{a\overline{b}}{l^{\frac{n_l}{2}}}m\right) = \psi\left(\frac{a\overline{b}}{l^{\frac{n_l}{2}}}m\right) = 1. \nonumber
\end{equation}
Furthermore, we use the reciprocity formula
\begin{equation}
	e\left(\frac{a\overline{b}}{l^{\frac{n_l}{2}}}m\right) = e\left(-\frac{a\overline{l^{\frac{n_l}{2}}}}{b}m\right)e\left(\frac{a}{bl^{\frac{n_l}{2}}}m\right) \nonumber
\end{equation}
to obtain 
\begin{equation}
	L_s = \sum_m \lambda_{\pi_0}(m) e\left(-\frac{a\overline{l^{\frac{n_l}{2}}}}{b}m\right)W_{\infty}\left(m\right)W_l(s;m), \nonumber
\end{equation}
for
\begin{equation}
	W_{\infty}(x) = e\left(\frac{a}{bl^{\frac{n_l}{2}}}x\right) F\left(\frac{x}{M}\right). \label{eq:def_Warch}
\end{equation}
This last formula for $L_s$ suits the application of our Vorono\"i summation formula which we will apply in the next subsection. Before we will continue we make the following observation.

\begin{lemma} \label{lm:periodicity_Wl}
The function $W_l(s;\cdot)$ is periodic modulo $l^{\frac{n_l}{2}-q-\abs{r}-k} \Z_l$.
\end{lemma}
\begin{proof}
First, observe that for $s=(a,b,k)\in S_0$ we have $k\leq q-\abs{r}$ and deduce
\begin{equation}
	\frac{n_l}{2}-q-\abs{r}-k \geq \frac{n_l}{2}-2q\geq \frac{n_l}{4}>0.\nonumber
\end{equation}
For $m\in \Z_l^{\times}[a,b,k]$ and $y \in \Z_l$ we argue as on \cite[p. 582]{BM15} to obtain
\begin{eqnarray}
	W_l(s;m+yl^{\frac{n_l}{2}-q-\abs{r}-k}) = \chi_l^{-1}(m)\psi_l\left(\frac{a\overline{b}}{l^{\frac{n_l}{2}}}m\right)\psi_l\left(\frac{-\alpha m^{-1}+a\overline{b}}{l^{q+\abs{r}+k}}y\right), \nonumber
\end{eqnarray}
One concludes using the definition of $\Z_l[a,b,k]$.
\end{proof}

\subsection{Applying the Vorono\"i formula}

In this section we will dualize the sum $L_s$ by  applying Theorem~\ref{th:vor_with_Cs}. The resulting expression will then be brought into good shape by using stationary phase arguments to evaluate the $p$-adic and archimedean Hankel transforms, see Lemma~\ref{lm:p-adic_stat} and Lemma~\ref{lm:arch_stat} below. The upshot is that we end up with a formula in a form which is suitable for extracting the necessary cancellation. This concludes step two described above. 

Combining Theorem~\ref{th:vor_with_Cs} with Lemma~\ref{lm:properties_p_adc_Besel} and Lemma~\ref{lm:periodicity_Wl} yields
\begin{equation}
	L_s = \sum_{c\geq -n_l+2q+2\abs{r}+2k}L_{s,c}, \nonumber
\end{equation}
for
\begin{eqnarray}
	L_{s,c} &=&\zeta_{N_1}(1) \frac{\eta(\pi_0,a,b)}{b_0b_2\sqrt{N_0}}\sum_{\du{\mu}\in \prod_{p\mid b_1}{}_p\mathfrak{X}_{v_p(b_1)}} \frac{\du{\mu}(\frac{ab_0b_2N_0N'_1(\du{\mu})}{b_1l^c})}{\sqrt{b_1N_1'(\du{\mu})}}\nonumber \\ 	
	&& \cdot\sum_{m_1\mid N_1^{\infty}}C(\pi_{0,N_1},\du{\mu},b_1,m_1) \lambda_{(\pi_0^{N_2})_{\du{\mu}}}\left(m_1\frac{N_1'(\du{\mu})}{b_1N_1}\right)\nonumber \\  
	&&\cdot \sum_{\substack{m\in\Z,\\(m,lN_1)=1}} e\left(l^cm_1m\frac{\overline{a N_0N_1}}{b_0b_2}\right) \lambda_{(\pi_0^{N_2})_{\du{\mu}}}\left(m
	\right)\mathcal{H}_{\sgn(m)}W_{\infty}\left(\frac{l^cm_1\abs{m}}{b_0^2b_2^2b_1N_1N_0}\right) \nonumber \\
	&&\qquad\qquad\qquad \qquad\qquad\qquad \qquad\qquad\qquad\cdot \mathcal{H}W_{l}\left(s;\frac{l^cm_1m}{b_0^2b_2^2b_1N_1N_0}\right).\nonumber
\end{eqnarray} 

We define
\begin{eqnarray}
	\mathcal{L}_{s,c}(m) &=& \sum_{\mu_l\in {}_l\mathfrak{X}_c'}\epsilon\left(\frac{1}{2},\mu_l^{-1}\right)^2\mu_l(mb^{-2}
 )^{-1}[\mathfrak{M}W_l(s;\cdot)](\mu_l), \nonumber  \\
	\mathcal{I}_{s,c}^{\pm}(m) &=& \int_0^{\infty}F\left(\frac{x}{M}\right)e\left(\frac{ax}{bl^{\frac{n_l}{2}}}\right)\mathcal{J}_{\kappa}^{\pm}\left(\frac{4\pi\sqrt{mx}}{bl^c}\right)dx \nonumber
\end{eqnarray}
This notation is taken from \cite{BM15}. However, the $l$-adic oscillatory function $\mathcal{L}_{s,c}$ differs slightly from the one given in \cite[(5.11)]{BM15}. This is due to the fact that we are working in the adelic setting which makes our function purely local. 

Note that according to the definition of $W_{\infty}$, \eqref{eq:def_Warch}, and the archimedean Hankel-transform, \eqref{eq:def_hankelarch}, we have
\begin{equation}
	\mathcal{I}_{s,c}^{\pm}(m) = \mathcal{H}_{\pm}W_{\infty}\left(\frac{m}{b^2l^{2c}}\right).\nonumber
\end{equation}

Further we have the following properties of the $p$-adic Hankel-transform.
\begin{lemma}[$p$-adic stationary phase] \label{lm:p-adic_stat}
If $-n_l+2q+2\abs{r}+2k\leq c< -2$, then
\begin{equation}
	\mathcal{H}W_{l}(s;yl^c) = \begin{cases}
		\omega_{\pi_{0},l}(l^{-\frac{c}{2}})l^{\frac{c}{2}}  \mathcal{L}_{s,-\frac{c}{2}}(yb^2) &\text{ if $c$ is even}, \\
		0 &\text{ else,}
	\end{cases}\nonumber
\end{equation}
for $y\in \Z_l^{\times}$.
\end{lemma}
\begin{proof}
Observe that for $\mu_l\neq 1$ we have $a(\mu_l\pi_{0,l})=2a(\mu_l)$ and $L(s,\mu_l\pi_{0,l})=1$. Thus, according to the definition of the $p$-adic Hankel-transform, \eqref{eq:def_p_adic_bessel}, and \eqref{eq:nice_support_for_simplereps}, we can write
\begin{align}
	\mathcal{H}W_{l}(s;yl^c) =& l^{\frac{c}{2}}\sum_{\substack{\mu_l\in {}_l\mathfrak{X}\setminus\{1\}, \\ c=-2a(\mu)} } \mu_l(y^{-1})\epsilon(\frac{1}{2},\mu_l^{-1}\pi_{0,l})[\mathcal{M}W_l(s;\cdot)](\mu_l) \nonumber \\
	&\quad + l^{\frac{c}{2}}B_{\pi_{0,l}}(l^{-c})[\mathcal{M}W_l(s;\cdot)](1). \nonumber
\end{align}
Since $\tilde{\pi}_{0,l}$ is unramified \eqref{eq:formula_unram_reps} shows that the trivial character contributes only for $c\geq 2$. Therefore, as long as $c<-2$ is even, we find
\begin{align}
	\mathcal{H}W_{l}(s;yl^c) =& \omega_{\pi_{0},l}(l^{-\frac{c}{2}})l^{\frac{c}{2}}\sum_{\mu_l\in {}_l\mathfrak{X}_{-\frac{c}{2}}' } \mu_l(y^{-1})\epsilon(\frac{1}{2},\mu_l^{-1})^2[\mathcal{M}W_l(s;\cdot)](\mu) \nonumber \\
	=&\omega_{\pi_{0},l}(l^{-\frac{c}{2}})l^{\frac{c}{2}}  \mathcal{L}_{s,-\frac{c}{2}}(yb^2), \nonumber
\end{align}
where we used that $\epsilon(\frac{1}{2},\mu_l\pi_{0,l}) = \omega_{\pi_{0},l}(l^{-\frac{c}{2}})\epsilon(\frac{1}{2},\mu_l)^2$.
\end{proof}

With this at hand, for $1< c\leq \frac{n_l}{2}-q-\abs{r}-k$, we can rewrite
\begin{eqnarray}
	L_{s,-2c} &=&\zeta_{N_1}(1) \frac{\eta(\pi,a,b)}{b_0b_2\sqrt{N_0}}\sum_{\du{\mu}\in \prod_{p\mid b_1}{}_p\mathfrak{X}_{v_p(b_1)}} \frac{\du{\mu}(ab_0b_1^{-1}b_2N_0N'_1(\du{\mu})l^{-2c})}{\sqrt{b_1N_1'(\du{\mu})}}\nonumber \\ 	
	&& \cdot \sum_{m_1\mid N_1^{\infty}}C(\pi_{0,N_1},\du{\mu},b_1,m_1)\lambda_{(\pi_0^{N_2})_{\du{\mu}}}\left(m_1\frac{N_1'(\du{\mu})}{b_1N_1}\right)\nonumber \\  
	&&\cdot \sum_{\substack{m\in\Z,\\(m,lN_1)=1}} e\left(m_1m\frac{\overline{a N_0N_1}}{b_0b_2l^{2c}}\right) \lambda_{(\pi_0^{N_2})_{\du{\mu}}}\left(m
	\right)\mathcal{I}_{s,c}^{\sgn(m)}\left(m_1\abs{m}\frac{b_1}{N_0N_1}\right)\nonumber \\
	&&\qquad\qquad \qquad \qquad\qquad \qquad \qquad \cdot \omega_{\pi_{0},l}(l^{c})l^{-c} \mathcal{L}_{s,c}\left(m_1m\frac{b_1}{N_0N_1}\right).\nonumber
\end{eqnarray} 

The oscillatory parts, $\mathcal{L}_{s,c}$ and $\mathcal{I}_{s,c}^{\pm}$, appearing in these sums have been evaluated in \cite{BM15}. Since we only shifted the argument we can reuse this evaluations. Recall \cite[Lemma~3, Lemma~4]{BM15}.

\begin{lemma}[Archimedean stationary phase]\label{lm:arch_stat}
The function $\mathcal{I}_{s,c}^{\pm}(m)$ is $O((l^{n_l}m)^{-100})$ unless
\begin{equation}
	m\ll l^{2c}\left(\frac{B^2 Z^2}{M}+\frac{A^2 M}{l^{n_l}}\right) l^{3\epsilon n_l} = \mathcal{M}l^{3\epsilon n_l}. \label{eq:range_I}
\end{equation}
In the range \eqref{eq:range_I} one has
\begin{equation}
	\mathcal{I}_{s,c}^{\pm}(m)=\left(\frac{\mathcal{M}l^{3\epsilon n_l}}{m}\right)^{\frac{1}{4}}\min\left(M,\frac{BZl^{\frac{n_l}{2}}}{A}\right)e(\theta_{s,c}m)W_{s,c}(m)+O(l^{-100n_l}), \nonumber
\end{equation}
where $W_{s,c}$ is smooth and satisfies
\begin{equation}
	x^j\frac{d^j}{dx^j}W_{s,c}(x)\ll_j l^{3\epsilon n_l}(Z^2l^{5\epsilon n_l})^j \nonumber
\end{equation}
and where
\begin{equation}
	\theta_{s,c}=\begin{cases}
		-\frac{l^{\frac{n_l}{2}-2c}}{ab} &\text{ if } \frac{AM}{BZ^2l^{\frac{n_l}{2}}}\geq 1,\\
		0 &\text{ else.}
	\end{cases} \nonumber
\end{equation}
\end{lemma}

\begin{lemma}
The function $\mathfrak{L}_{s,c}$ evaluates to
\begin{equation}
	\mathcal{L}_{s,c}(m) = \begin{cases}
		\gamma l^{\frac{c}{2}-\frac{n_l}{4}}\chi_l^{-1}(\overline{a}b)\sum_{\pm}\Phi_c^{\pm}(\frac{m}{ab}) &\text{ if }\frac{\alpha bm}{a}\in (\Z_l^{\times})^2, \\
		0&\text{ else.}
	\end{cases} \nonumber
\end{equation}
where $\gamma$ is a constant of absolute value $1$ which depends only on the parity of $\frac{n_l}{2}$, and 
\begin{multline}
	\Phi_c^{\pm}(x) = \epsilon(\pm(\alpha x)_{\frac{1}{2}},p^{\rho})\chi_l^{-1}(\alpha+\frac{1}{2}p^{2(n-c)}x\pm p^{n-c}(\alpha x+\frac{1}{4}p^{2(n-c)}x^2)_{\frac{1}{2}}) \\
	 \cdot \psi_p(-\frac{1}{p^c}(\frac{1}{2}p^{n-c}x\pm(\alpha x+\frac{1}{4}p^{2(n-c)}x^2)_{\frac{1}{2}})). \label{eq:l-adic_oscilation}
\end{multline}
\end{lemma}
\begin{proof}
The computations are essentially the same as in \cite[Section~7.2]{BM15}. Thus let us show how to de-adelize $\mathcal{L}_{s,c}$. This reduces the statement directly to \cite[Lemma~4]{BM15}.

Note that there is a one to one correspondence between primitive Dirichlet characters $\xi$ with conductor $l^c$ and $\mu_l\in{}_l\mathfrak{X}_c'$. The natural choice is such that the adelization of $\xi$ is $\chi_{\mu_l}^{-1}$. In particular, we have $\xi(x) = \mu_l^{-1}(x)$ for all $(x,l)=1$. Furthermore, $\tau(\xi)=l^{\frac{c}{2}}\epsilon(\frac{1}{2},\mu_l^{-1})$. In view of Lemma~\ref{lm:periodicity_Wl} we find that
\begin{equation}
	[\mathcal{M}W_l(s;\cdot)](\mu_l)  = \frac{\zeta_l(1)}{l^{\frac{n_l}{2}-q-\abs{r}-k}}\hat{f}_s(\xi), \nonumber
\end{equation}
for the to $\mu_l$ corresponding Dirichlet character $\xi$. Here we temporarily use some notation from \cite{BM15}. In particular, the definition of $\hat{f}_s$ from \cite[(5.6)]{BM15}. Thus we have
\begin{equation}
	\mathcal{L}_{s,c}(m)=\frac{\zeta_l(1)}{l^{\frac{n_l}{2}-q-\abs{r}-k}} \sum_{\substack{\xi \mod l^c, \\\xi \text{ primitive}}}\frac{\tau(\xi)^2}{l^c}\xi(b^2\overline{m})\hat{f}_s(\xi). \nonumber
\end{equation} 
for $(m,l)=1$. To the remaining sum we can apply \cite[Lemma~4]{BM15} directly. The result follows after translating back to our notation.

\end{proof}

Combining everything we have
\begin{equation}
	L_s= \sum_{1< c \leq \frac{n_l}{2}-q-\abs{r}-k} L_{s,-2c} + E, \nonumber
\end{equation}
where $E$ collects the vales of $c$ together that we neglected till so far. More precisely,
\begin{equation}
	E=\sum_{c\geq -2} L_{s,c}. \nonumber
\end{equation} The following estimate for the error $E$ can be understood as a truncation in the $l$-aspect. It is closely related to the contributions treated in the beginning of \cite[Section~5.3]{BM15}.

\begin{lemma}
Under our working assumptions we have
\begin{equation}
	E \ll_{\pi_0,j}\zeta_l(1) M l^{-q-\abs{r}-k}\left( 1+\frac{l^{1+q-r}}{\sqrt{M}} \right)^{j+\frac{1}{2}}\left(\frac{Zl^{1+q-r}}{\sqrt{M}}+\frac{l^{1-q+r}\sqrt{M}}{l^{\frac{n_l}{2}}}\right)^{j}. \nonumber
\end{equation} 
\end{lemma}
\begin{proof}
By taking absolute values into the sum defining $L_{s,c}$ and using \eqref{eq:estimat_C} together with some other trivial estimates we get
\begin{eqnarray}
	E &=& \sum_{c\geq -2} L_{s,c} \nonumber \\
	&\ll_{N}& \frac{1}{b}\sum_{\du{\mu}\in \prod_{p\mid b_1}{}_p\mathfrak{X}_{v_p(b_1)}}\sum_{m\in \Z_{\neq 0}}(m,N_1^{\infty})^{\frac{7}{64}}\abs{\lambda_{(\pi_0^{N_2})_{\du{\mu}}}\left(\frac{mN_1'(\du{\mu})}{(m,l^{\infty})b_1N_1}\right)} \nonumber \\
	&&\qquad \cdot \abs{\mathcal{H}_{\sgn(m)}W_{\infty}\left(\frac{\abs{m}}{l^2b_0^2b_2^2b_1N_1N_0}\right)} \abs{\mathcal{H}W_{l}\left(s;\frac{m}{l^2b_0^2b_2^2b_1N_1N_0}\right)}. \nonumber
\end{eqnarray}
As described in \cite[Section~5.3, (2.10)]{BM15} we have the bound
\begin{eqnarray}
	\mathcal{H}_{\sgn(m)}W_{\infty}\left(\frac{\abs{m}}{l^2b_0^2b_2^2b_1N_1N_0}\right) \ll_{\pi_{0,\infty},N,j} M\left( 1+\frac{lb}{\sqrt{M\abs{m}}} \right)^{j+\frac{1}{2}}\left(\frac{Z_0 l b}{\sqrt{M\abs{m}}}\right)^{j}, \nonumber
\end{eqnarray}
for $Z_0 = Z+AMB^{-1}l^{-\frac{n_l}{2}}$. From \eqref{eq:def_p_adic_bessel} we deduce the trivial estimate
\begin{equation}
	\mathcal{H}W_l\left( s;\frac{m}{l^2b_0^2b_2^2b_1N_1N_0}\right) \ll (m,l^{\infty})^{\frac{1}{2}}\sup_{\mu_l}\abs{[\mathfrak{M}W_l(s;\cdot)](\mu_l)}. \nonumber
\end{equation}
Furthermore, according to \cite{BB11}, we can always estimate 
\begin{equation}
\abs{\lambda_{(\pi_0^{N_2})_{\du{\mu}}}\left(\frac{mN_1'(\du{\mu})}{(m,l^{\infty})b_1N_1}\right)} \ll_{N,\epsilon} \left(\frac{m}{(m,l^{\infty})}\right)^{\frac{7}{64}+\epsilon}.\nonumber
\end{equation}
Using these estimates together with  
\begin{eqnarray}
	w\leq pb\leq 2pB\leq p^{k+2r^-+1}, \quad k\leq l-\abs{r}, \quad -\abs{r}+2r^{\pm} = \pm r,\quad  \frac{w}{B}\leq 2p  \nonumber 
\end{eqnarray}
yields 
\begin{eqnarray}
	E &\ll_{\pi_0,j}&M\left( 1+\frac{l^{1+l-r}}{\sqrt{M}} \right)^{j+\frac{1}{2}}\left(\frac{Zl^{1+l-r}}{\sqrt{M}}+\frac{l^{1-l+r}\sqrt{M}}{l^{\frac{n_l}{2}}}\right)^{j}\nonumber \\
	&&\qquad\qquad\qquad\qquad\qquad\qquad\cdot \sup_{\mu_l}\abs{[\mathfrak{M}W_l(s;\cdot)](\mu_l)}.\nonumber
\end{eqnarray}
The result follows by estimating $\abs{[\mathfrak{M}W_l(s;\cdot)](\mu_l)}$ trivially.
\end{proof}

\subsection{Extracting cancellation on average}

In this final section we perform the third and fourth step as described above. We start by introducing some more notation. This will make it easier to see how the Cauchy-Schwarz inequality is applied. We set
\begin{eqnarray}
	\mathcal{M}' &=& \mathcal{M}\frac{N_0N_1}{b_1} \text{ and }\nonumber \\
	\tilde{W}_{s,c,\du{\mu}} &=& \frac{B}{b}\chi_l^{-1}(\overline{a}b)\du{\mu}(ab_0b_1^{-1}b_2N_0N'_1(\du{\mu})l^{2c})\eta(\pi_0,a,b)W_{s,c}. \nonumber
\end{eqnarray}
In particular we have $\abs{\tilde{W}_{s,c,\du{\mu}}}\asymp_f \abs{W_{s,c}}$. Furthermore, we define
\begin{eqnarray}
	\kappa_s(x,\du{\mu}) &=& e\left(\left( l^c\frac{\overline{a N_0N_1}}{b_0b_2}+\frac{\theta_{s,c}b_1}{N_0N_1} \right)x\right)  \tilde{W}_{s,c,\du{\mu}}\left(\frac{xb_1}{N_0N_1} \right) \sum_{\pm}\Phi_c^{\pm}\left(\frac{xb_1}{N_0N_1ab}\right). \nonumber
\end{eqnarray}
Inserting the results from the previous subsection and dealing with the error terms in the obvious way leads to
\begin{equation}
	L \ll_{\pi_0} l^{\epsilon n_l}\max_{\substack{0\leq k\leq q-\abs{r},\\ 1< c\leq \frac{n_l}{2}-q-\abs{r}-k}}\max_{\substack{A\leq l^{k+2r^+}, \\ B\leq l^{k+2r^-},\\ b_1\mid N, \\ \du{\mu}\in \prod_{p\mid b_1 }{}_p\mathfrak{X}_{v_p(b_1)}}}\left(\abs{L_{A,B,b_2,k,c, \du{\mu}}}+(ABl^{\epsilon n_l})E+l^{-50n_l}\right), \label{eq:end_Step_2}
\end{equation}
for
\begin{eqnarray}
	L_{A,B,b_2,k,c,\du{\mu}} &=& \frac{\min(M,\frac{BZl^{\frac{n_l}{2}}}{A})}{l^{\frac{n_l+2c}{4}}B} \sum_{\substack{r\mid b_1^{\infty},\\r\leq l^{3\epsilon n_l} \mathcal{M}'}}\sum_{\substack{(m,lb_1)=1,\\m\leq \frac{l^{3\epsilon n_l} \mathcal{M}'}{r},}} C(\pi_{0,N_1},\du{\mu},b_1,r)\nonumber \\
	&& \lambda_{(\pi_0^{N_2})_{\du{\mu}}}\left(mr\frac{N_1'(\du{\mu})}{b_1N_1}\right)\left(\frac{ \mathcal{M}'l^{3\epsilon n_l}}{\abs{m}r}\right)^{\frac{1}{4}}\sum_{\substack{s=(a,b,k)\in S^{\circ},\\ b=b_0b_1b_2,\\ A\leq \abs{a}< 2A,\\ B\leq b< 2B,\\ \alpha\frac{b_0b_2mr}{aN_0}\in(\Z_l^{\times})^2}} \kappa_s(mr,\du{\mu}). \nonumber
\end{eqnarray}
Finally, we define
\begin{equation}
	\Xi_{s_1,s_2,\du{\mu}} = \sum_{\substack{1\leq m\leq \mathcal{M}'q^{3\epsilon}, \\ \alpha \frac{b^{(j)}m}{a^{(j)}b_1N_0N_1}\in(\Z_l^{\times})^2 } } \kappa_{s_1}(m,\du{\mu})\overline{\kappa_{s_2}(m,\du{\mu})}.\nonumber
\end{equation}

An application of the Cauchy-Schwarz inequality together with the Ranking-Selberg bound yields
\begin{equation}
	L_{A,B,b_2,k,c} \ll_{N} M^{\frac{1}{2}}Zl^{2\epsilon n_l}p^{\frac{2c-n_l}{4}}\left(\sum_{\substack{s_1,s_2\in S^{\circ},\\ A\leq a^{(j)}< 2A, \\ B\leq b^{(j)}< 2B,\\ b^{(j)}=b_0^{(j)}b_1b_2^{(j)}}} \abs{\Xi_{s_1,s_2,\du{\mu}}}\right)^{\frac{1}{2}}. \label{eq:part_of_step3} 
\end{equation}
This is similar to \cite[(5.17)]{BM15} and completes step three. Indeed, the application of Cauchy-Schwarz successfully removed the automorphic weights and we are left with an average of exponential sums. We still need to bound this average non-trivially.

Moving on to the final step we have to bound the exponential sums $\Xi_{s_1,s_2,\du{\mu}}$. To do so we adapt \cite[Lemma~5]{BM15} to our situation. Indeed, we will be able to reduce this directly to the situation treated in \cite{BM15}. Note that the key estimate comes from a second derivative $p$-adic van der Corput estimate given in \cite[Theorem~5]{BM15}. This goes as follows.
\begin{lemma} \label{lm:after_vdC}
Under the usual assumptions  we have
\begin{equation}
	\Xi_{s_1,s_2,\du{\mu}} = \sum_{\Omega\in\{0,\text{ord}_p(a^{(1)}b^{(1)}-a^{(2)}b^{(2)}\}} \Xi_{s_1,s_2,\du{\mu},\Omega}. \label{eq:Xi_decomp}
\end{equation}
Furthermore,
\begin{equation}
	\Xi_{s_1,s_2,\du{\mu},\Omega} \ll_N \begin{cases}
		q^{17\epsilon} Z^2 p (p^{\frac{\Omega-c}{2}}\mathcal{M}+p^{\frac{c-\Omega}{2}}) &\text{ if } \Omega\leq c-2,\\
		q^{9\epsilon}\mathcal{M} &\text{ if }c-1\leq \Omega \leq \infty.
	\end{cases} \nonumber
\end{equation}
\end{lemma}
In the proof we closely follow \cite[Section~9]{BM15}.
\begin{proof}
Let $v=\frac{n_l}{2}-c$, $\epsilon = \pm 1$ and $x\in \alpha (\Z_l^{\times})^2$. Then
\begin{equation}
	\tilde{\Phi}_v^{\epsilon}(x)= \chi_l^{-1}(\alpha+\frac{1}{2}l^{2v}x+\epsilon l^v(\alpha x+\frac{1}{4}l^{2v}x^2)_{\frac{1}{2}})\psi_l(-\frac{1}{l^c}(\frac{1}{2}l^vx+\epsilon(\alpha x+\frac{1}{4}l^{2v}x^2)_{\frac{1}{2}})).\nonumber
\end{equation} 
Further, we define
\begin{equation}
	\boldsymbol{\Phi}_{\mathbf{s},c}^{\mathbb{\epsilon}}(x) = \tilde{\Phi}_{\frac{n_l}{2}-c}^{\epsilon_1}\left(\frac{xb_1}{a^{(1)}b^{(1)}N_0N_1}\right)\overline{\tilde{\Phi}_{\frac{n_l}{2}-c}^{\epsilon_2}\left(\frac{xb_1}{a^{(2)}b^{(2)}N_0N_1}\right)} \nonumber
\end{equation}
where $\mathbf{s}=(s_1,s_2)$ and $\boldsymbol{\epsilon}=(\epsilon_2,\epsilon_2)$. We also set
\begin{eqnarray}
	\mathbf{W}_{\mathbf{s},c,\du{\mu}}(x) &=& \tilde{W}_{s_1,c,\du{\mu}}\left(x \frac{b_1}{N_0N_1}\right) \overline{\tilde{W}_{s_2,c,\du{\mu}}\left(x \frac{b_1}{N_0N_1}\right) }, \nonumber\\
	\boldsymbol{\epsilon}_{\mathbf{s},c}^{\boldsymbol{\epsilon}}(x) &=& \epsilon(\epsilon_1(\alpha m b_1\overline{N_0N_1a^{(1)}b^{(1)}})_{\frac{1}{2}},p^{\rho})\overline{\epsilon}(\epsilon_2(\alpha m b_1\overline{N_0N_1a^{(2)}b^{(2)}})_{\frac{1}{2}},p^{\rho}) \text{ and } \nonumber \\
	w_{\mathbf{s},c} &=& \left( l^c\frac{b_1\overline{a N_0N_1}}{b^{(1)}}+\frac{\theta_{s_1,c}b_1}{N_0N_1} \right)-\left(  l^c\frac{b_1\overline{a N_0N_1}}{b^{(2)}}+\frac{\theta_{s_2,c}b_1}{N_0N_1} \right).\nonumber
\end{eqnarray}

We can also assume that $a^{(1)}b^{(1)}a^{(2)}b^{(2)} \in(\Z_l^{\times})^2$ since otherwise the two conditions in the $s_1$, $s_2$ sum can not be satisfied simultaneously. Under this condition, and in the new notation we have
\begin{equation}
	\Xi_{s_1,s_2,\du{\mu}} = \sum_{\boldsymbol{\epsilon}\in\{\pm 1\}^2}\sum_{\substack{m\leq \mathcal{M}',\\ \alpha \frac{b^{(1)}m}{a^{(1)}b_1N_0N_1}\in(\Z_l^{\times})^2}} \boldsymbol{\epsilon}_{\mathbf{s},c}^{\boldsymbol{\epsilon}}(m)e(w_{\mathbf{s},c}m)\boldsymbol{\Phi}_{\mathbf{s},c}^{\mathbb{\epsilon}}(m)\mathbf{W}_{\mathbf{s},c,\du{\mu}}(m). \nonumber
\end{equation}
It is clear that \cite[Lemma~13]{BM15} holds also in our case. This is because our $\boldsymbol{\Phi}_{\mathbf{s},c}^{\mathbb{\epsilon}}$ is simply a shift of the one considered in the reference. Furthermore, all the necessary assumptions are in place to make this work. The decomposition \eqref{eq:Xi_decomp} is as in \cite{BM15} and is obvious from the result \cite[Lemma~13]{BM15}.

We note that $\text{ord}_l(b_1((N_0N_1)^{-1})=0$ so that we can continue exactly as in \cite{BM15}. Indeed, \cite[Lemma~13]{BM15} provides us with the necessary conditions to apply \cite[Theorem~5]{BM15} when $\Omega\leq c-2$. In the re remaining cases we estimate trivially. After discarding possible factors coming from the shift in the archimedean factor we obtain the desired bounds.   
\end{proof}

Finally, we note that the bounds for $\Xi_{s_1,s_2,\du{\mu},\Omega}$ as well as the $\Omega$-decomposition are independent of $\du{\mu}$ and $b_1$. Thus, we can follow exactly the argument from \cite[Section~5.5]{BM15}. Note that this includes assuming
\begin{equation}
	Z^5l^{\frac{7}{3}}l^{\frac{2n_l}{3}}\ll_N M \ll_N (l^{\frac{n_l}{2}}Z)^{1+\epsilon}. \nonumber
\end{equation} As in \cite[Section~5.1]{BM15} one can justify that it is enough to treat this range for $M$. Finally one arrives at
\begin{equation}
	L\ll_{\pi_0} M^{\frac{1}{2}}Z^{\frac{5}{2}}l^{\frac{7}{6}}l^{\frac{n_l}{6}+11\epsilon n_l} \nonumber
\end{equation}
and the proof of Theorem~\ref{th:main_th_2} is complete.

The statement of Theorem~\ref{th:main_th_1} follows from Theorem~\ref{th:main_th_2} using standard arguments including adelization, approximate functional equation and partitions of unity.

\appendix

\section{Tables for $c_{t,l}(\mu)$} \label{app:ctlv}

In this appendix we recall some results from \cite{As17_1}. We will state them in a notation which is suitable for applications in the setting of paper.

Throughout this section we are dealing with smooth irreducible unitary generic representations $\pi_{p}$ of $GL_2(\Q_{p})$. To such a representation we attach the local Whittaker new vector $W_{\pi,p}$, normalized by $W_{\pi,p}(1)=1$. We have the expansion
\begin{equation}
	W_{\pi,p}(g_{t,l,v}) = \sum_{\mu_p\in {}_p\mathfrak{X}_l}c_{t,l}(\mu_p)\mu_p(v). \nonumber
\end{equation}
These constants have been computed in \cite[Section~2]{As17_1}. In the following we define new constants via
\begin{equation}
	c_{t,l}(\mu_p) = c_{p}(\pi_{p},l,t,\mu_p)\zeta_{p}(1)p^{-\frac{l+t+a(\mu_p\pi_p)}{2}}\lambda_{\chi_{\mu_p}\pi}(p^{t+a(\mu_p\pi_{p})+\delta_{\mu_p\pi_{p}}}),  \nonumber
\end{equation}
for some $\delta_{\mu\pi_{p}}\in \N$ which in most cases turn out to be the degree of the Euler-factor of $\mu\pi_{p}$.

In the following subsections we give evaluations of the constants for each possible representation focusing on the non-zero cases. As a result we obtain the bound
\begin{equation}
	\abs{c_{p}(\pi_{p},l,t,\mu_p)} \leq 5p^{\frac{1}{2}}t\max_{i=1,2}(\abs{\alpha_i}^t), \label{eq:bound_cp}
\end{equation}
for $\alpha_i = \chi_i(\varpi_{p})$ if $\pi_p = \chi_1\boxplus\chi_2$ and $\alpha_i =1$ otherwise. Note that, since $\pi_p$ is unitary, we have $\abs{\alpha_i}=1$ except for $\chi_1$ equals $\chi_2$ up to unramified twist. The latter cannot be excluded without assuming the Ramanujan conjecture. Indeed, such representations might arise as components of twists of Maa\ss\  forms failing the Ramanujan conjecture.

\subsection{Supercuspidal representations}

Recall that in this case $\lambda_{\chi_{\mu_p}\pi_p}(p^m)=\delta_{m=0}$ and $\delta_{\mu_p\pi_p}=0$ for all $\mu_p$. Thus from \cite[Section~2.1]{As17_1} we extract the following.

\begin{equation}\begin{array}{| c |c | c | c |}
	\hline
	c_{p}(\pi_p,l,t,\mu_p) & \mu_p = 1 & \mu_p\in {}_p\mathfrak{X}_l\setminus\{1\} \\ 
	\hline
	l=0 & \epsilon(\frac{1}{2},\tilde{\pi}_p)\zeta_{p}(1)^{-1} & -  \\
	\hline
	l=1 & -p^{-\frac{1}{2}}\epsilon(\frac{1}{2},\tilde{\pi}_p) & \epsilon(\frac{1}{2},\mu_p)\epsilon(\frac{1}{2},\mu^{-1}_p\tilde{\pi}_p) \\
	\hline
	l>1 & 0 & \epsilon(\frac{1}{2},\mu_p)\epsilon(\frac{1}{2},\mu^{-1}_p\tilde{\pi}_p)  \\
	\hline
\end{array}\nonumber\end{equation}

\subsection{Twists of Steinberg}

Here we consider $\pi_{p}=\chi St$ for some ramified character $\chi$. We have 
\begin{equation}
	\lambda_{\chi_{\mu_p}\pi_p}(p^m) =\begin{cases}
		\delta_{m=0} &\text{ if } \mu_p \neq \chi^{-1}, \\
		q^{-\frac{m}{2}}\delta_{m\geq 0} &\text{ if } \mu_p = \chi^{-1}
	\end{cases} \nonumber
\end{equation}
and $\delta_{\mu_p\pi_p}=1$ if $\mu_p=\chi^{-1}$ and $0$ otherwise. As in \cite[Lemma~2.1]{As17_1} one obtains the following evaluations.

\begin{equation}\begin{array}{| c |c | c | c | c |}
	\hline
	c_{p}(\pi_p,l,t,\mu_p) & \mu_p = 1 & \mu_p = \chi^{-1} & \mu_p \in {}_p\mathfrak{X}'\setminus \{1,\chi^{-1}\}  \\
	\hline
	l=0 &\epsilon(\frac{1}{2},\tilde{\pi}_p)\zeta_{p}(1)^{-1} &-&- \\
	\hline
	l=1  & -\epsilon(\frac{1}{2},\tilde{\pi}_p)p^{-\frac{1}{2}} & \epsilon(\frac{1}{2},\mu_p)p^{-\frac{3}{2}} \text{ if } t\leq -2 & \epsilon(\frac{1}{2},\mu_p^{-1}\tilde{\pi}_p)\epsilon(\frac{1}{2},\mu_p)\\
	&& -\epsilon(\frac{1}{2},\mu_p)\frac{p^{\frac{1}{2}}}{\zeta_{p}(2)} \text{ if } t>-2 &  \\
	\hline
	l>1  & 0 & \epsilon(\frac{1}{2},\mu_p)p^{-\frac{3}{2}} \text{ if } t\leq -2 & \epsilon(\frac{1}{2},\mu^{-1}_p\tilde{\pi}_p)\epsilon(\frac{1}{2},\mu_p) \\
	&& -\epsilon(\frac{1}{2},\mu_p)\frac{p^{\frac{1}{2}}}{\zeta_{p}(2)} \text{ if } t>-2 &  \\
	\hline
\end{array}\nonumber\end{equation}

\subsection{ Irreducible principal series}

In this section we treat three cases. First, we look at $\pi_p = \chi_1  \boxplus \chi_2 $ with $\chi_1\vert_{\mathcal{O}^{\times}}\neq \chi_2\vert_{\mathcal{O}^{\times}}$. In this case $\delta_{\mu_p\pi_p}= 1$ if $\mu_p\vert_{\Z_p^{\times}} =\chi_i^{-1}\vert_{\Z_p^{\times}}$ and $0$ otherwise. Furthermore,
\begin{equation}
	\lambda_{\mu_p\pi_p}(p^m) =\begin{cases}
		\delta_{m=0} &\text{ if } \mu_p\vert_{\Z_p^{\times}} \neq \chi_i^{-1}\vert_{\Z_p^{\times}}, \\
		\chi_i(p^m)\delta_{m\geq 0} &\text{ if } \mu_p\vert_{\Z_p^{\times}} = \chi_i^{-1}\vert_{\Z_p^{\times}}
	\end{cases} \nonumber
\end{equation}

The following table can be deduced from \cite[Lemma~2.2]{As17_1}.

\begin{equation}\begin{array}{| c |c | c | c | c |}
	\hline
	c_{p}(\pi_{p},l,t,\mu_p) & \mu_p = 1 & \mu_p\vert_{\Z_p^{\times}} = \chi_i^{-1}\vert_{\Z_p^{\times}} & \mu_p \in {}_p\mathfrak{X}'\setminus \{1,\chi_i^{-1}\}  \\
	\hline
	l=0 &\epsilon(\frac{1}{2},\tilde{\pi}_p)\zeta_{p}(1)^{-1} &-&-\\
	\hline
	l=1  & -\epsilon(\frac{1}{2},\tilde{\pi}_p)p^{-\frac{1}{2}} & -\epsilon(\frac{1}{2},\mu_p^{-1}\tilde{\pi}_p)\epsilon(\frac{1}{2},\mu_p)\chi_i^{-1}(p)p^{-1} & \epsilon(\frac{1}{2},\mu^{-1}_p\tilde{\pi}_p)\epsilon(\frac{1}{2},\mu_p)  \\
	&&\qquad  \text{ if } t\leq -a(\mu_p\pi_p)-1 &\\
	&& \epsilon(\frac{1}{2}\mu^{-1}_p\tilde{\pi}_p)\epsilon(\frac{1}{2},\mu_p)\chi_i^{-1}(p)\zeta_{p}(1)^{-1}&  \\
	&& \qquad  \text{ if } t> -a(\mu_p\pi_p)-1 &\\
	\hline
	l>1  & 0 &  -\epsilon(\frac{1}{2},\mu_p^{-1}\tilde{\pi}_p)\epsilon(\frac{1}{2},\mu_p)\chi_i^{-1}(p)p^{-1} & \epsilon(\frac{1}{2},\mu_p^{-1}\tilde{\pi}_p)\epsilon(\frac{1}{2},\mu_p) \\
	&&\qquad \text{ if } t\leq -a(\mu_p\pi_p)-1 &\\
	&& \epsilon(\frac{1}{2}\mu^{-1}_p\tilde{\pi}_p)\epsilon(\frac{1}{2},\mu_p)\chi_i^{-1}(p)\zeta_{p}(1)^{-1}&  \\
	&&\qquad  \text{ if } t> -a(\mu_p\pi_p)-1 &\\
	\hline
\end{array}\nonumber\end{equation}

Next we look at $\pi_{p}= \chi_1\boxplus\chi_2$ where $\chi_1\vert_{\Z_p^{\times}}=\chi_2\vert_{\Z_p^{\times}}$.  In this case $\delta_{\mu_p\pi_p}= 2$ if $\mu\vert_{\Z_p^{\times}} =\chi_1^{-1}\vert_{\Z_p^{\times}}$ and $0$ otherwise. Furthermore,
\begin{equation}
	\lambda_{\mu_p\pi_p}(p^m) =\begin{cases}
		\delta_{m=0} &\text{ if } \mu_p\vert_{\Z_p^{\times}} \neq \chi_1^{-1}\vert_{\Z_p^{\times}}, \\
		\frac{\chi_1(p^{m+1})-\chi_2(p^{m+1})}{\chi_1(p)-\chi_2(p)}\delta_{m\geq 0} &\text{ if } \mu_p\vert_{\Z_p^{\times}} = \chi_1^{-1}\vert_{\Z_p^{\times}}
	\end{cases} \nonumber
\end{equation}

Using \cite[Lemma~2.2]{As17_1} we produce the following table.

\begin{equation}\begin{array}{| c |c | c | c | c |}
	\hline
	c_{p}(\pi_{p},l,t,\mu_p) & \mu_p = 1 & \mu_p\vert_{\Z_p^{\times}} = \chi_1^{-1}\vert_{\Z_p^{\times}} & \mu_p \in {}_p\mathfrak{X}'\setminus \{1,\chi_1^{-1}\}  \\
	\hline
	l=0 &\epsilon(\frac{1}{2},\tilde{\pi}_p)\zeta_{p}(1)^{-1} &-&- \\
	\hline
	l=1  & -\epsilon(\frac{1}{2},\tilde{\pi}_p)p^{-\frac{1}{2}} & \epsilon(\frac{1}{2},\mu_p)p^{-2} & \epsilon(\frac{1}{2},\mu_p^{-1}\tilde{\pi}_p)\epsilon(\frac{1}{2},\mu_p)  \\
	&&\qquad  \text{ if } t\leq -2 &\\
	&& -\epsilon(\frac{1}{2},\mu_p)\frac{p^{-1}}{\zeta_{p}(1)}&  \\
	&& \qquad  \text{ if } t= -1 & \\
	&& \epsilon(\frac{1}{2},\mu_p)(\frac{1+p^{-1}-p^{-2}}{\zeta_{p}(1)^2}\frac{\lambda_{\mu_p\pi_p}(p^{t})}{\lambda_{\mu_p\pi_p}(p^{t+2})}-\zeta_{p}(1)^{-1})&  \\
	&&\qquad  \text{ if } t\geq 0 & \\
	\hline
	l>1  & 0 &  \epsilon(\frac{1}{2},\mu_p)p^{-2} & \epsilon(\frac{1}{2},\mu^{-1}_p\tilde{\pi}_p)\epsilon(\frac{1}{2},\mu_p)  \\
	&&\qquad \text{ if } t\leq -2 &\\
	&&-\epsilon(\frac{1}{2},\mu_p)\frac{p^{-1}}{\zeta_{p}(1)}&  \\
	&&\qquad  \text{ if } t=-1 & \\
	&& \epsilon(\frac{1}{2},\mu_p)(\frac{1+p^{-1}-p^{-2}}{\zeta_{p}(1)^2}\frac{\lambda_{\mu_p\pi_p}(p^{t})}{\lambda_{\mu_p\pi_p}(p^{t+2})}-\zeta_{p}(1)^{-1})&  \\
	&&\qquad  \text{ if } t\geq 0 & \\
	\hline
\end{array}\nonumber\end{equation}

Finally, we need to look at $\pi_{p} = \chi_1\boxplus\chi_2$ with $a(\chi_1)>a(\chi_2)=0$.  In this case we have
\begin{equation}
	\lambda_{\mu_p\pi_p}(p^m) =\begin{cases}
		\delta_{m=0} &\text{ if } \mu_p \neq \omega_{\pi,p}^{-1}, \\
		\chi_2(p^m)\delta_{m\geq 0} &\text{ if } \mu_p = \omega_{\pi,p}^{-1}.
	\end{cases} \nonumber
\end{equation}
Also, $\delta_{\mu_p\pi_p}=1$ if $\mu_p=\omega_{\pi,p}^{-1}$ and $0$ otherwise. For technical reasons we also put $\delta_{\pi_p} = l$. From \cite[Lemma~2.3]{As17_1} one gets the following results.

\begin{equation}\begin{array}{| c |c | c | c | c |}
	\hline
	c_{p}(\pi_{p},l,t,\mu_p) & \mu_p = 1 & \mu_p = \omega_{\pi,p}^{-1} & \mu_p \in {}_p\mathfrak{X}'\setminus \{1,\omega_{\pi,p}^{-1}\}  \\
	\hline
	l=0 &\epsilon(\frac{1}{2},\tilde{\pi}_p)\zeta_{p}(1)^{-1} &-&-\\
	\hline
	l>1  & \epsilon(\frac{1}{2},\tilde{\pi}_p)\chi_1(p^l)p^{-\frac{l}{2}} &  -\omega_{\pi,p}(-1)\chi_2(p^{1-l})p^{-1} & \epsilon(\frac{1}{2},\mu_p^{-1}\tilde{\pi}_p)\epsilon(\frac{1}{2},\mu_p)  \\
	&&\qquad \text{ if } t\leq -a(\mu_p\pi_p)-1 &\\
	&& \omega_{\pi,p}(-1)\chi_2(p^{1-l})&  \\
	&&\qquad  \text{ if } t> -a(\mu_p\pi_p)-1 & \\
	\hline
\end{array}\nonumber\end{equation}

\bibliographystyle{plain}
\bibliography{bibliography} 

\end{document}